\documentclass[a4paper,11pt,reqno]{amsart}
\usepackage[utf8]{inputenc}
\usepackage[T1]{fontenc}
\usepackage{lmodern}
\usepackage[english]{babel}
\usepackage{amsmath,a4wide}
\usepackage{bbm}
\usepackage{xfrac}
\usepackage{esint}
\usepackage{stmaryrd,mathrsfs,bm,amsthm,mathtools,yfonts,amssymb,color,braket,booktabs,graphicx,graphics,amsfonts}
\usepackage{latexsym,microtype,indentfirst,hyperref}
\usepackage{xcolor}
\usepackage{courier}
\usepackage[colorinlistoftodos]{todonotes}
\usepackage{float}
\usepackage{tikz}
\usepackage{enumerate}

\begingroup
\newtheorem{theorem}{Theorem}[section]
\newtheorem{lemma}[theorem]{Lemma}
\newtheorem{proposition}[theorem]{Proposition}
\newtheorem{corollary}[theorem]{Corollary}

\endgroup

\theoremstyle{definition}
\newtheorem{definition}[theorem]{Definition}
\newtheorem{remark}[theorem]{Remark}
\newtheorem{assumption}[theorem]{Assumption}

\newtheorem{resultado}{Theorem}

\newcommand\res{\mathop{\hbox{\vrule height 7pt width .3pt depth 0pt\vrule height .3pt width 5pt depth 0pt}}\nolimits}

\newcommand{\RR}{\mathbb{R}}
\newcommand{\CC}{\mathbb{C}}

\newcommand{\HH}{\mathcal{H}}

\newcommand{\dist}{\text{dist}}

\newcommand{\bB}{\mathbf{B}}
\usepackage{xcolor}

\def\a#1{\left\llbracket{#1}\right\rrbracket}
\numberwithin{equation}{section}

\title[Multiplicity $2$ essential boundary singularity of Allard-type]{An essential one sided boundary singularity for a $3$-dimensional area minimizing current in $\RR^5$}

\author[I. Fleschler]{Ian Fleschler}
\address{Department of Mathematics, Fine Hall, Princeton University, Washington Road, Princeton, NJ 08540, USA}
\email{imf@princeton.edu}

\begin{document}
    \begin{abstract}
We construct a $3$-dimensional area minimizing current $T$ in $\RR^5$ whose boundary contains a real analytic surface of multiplicity $2$ at which $T$ has a density $1$ essential boundary singularity with a flat tangent cone. This example shows that the boundary regularity theory we developed with Reinaldo Resende in another paper, which extends Allard’s classical boundary regularity result to higher boundary multiplicity, is dimensionally sharp.

The construction of $T$ relies on the prescription of boundary data with non-trivial topology, which makes it a flexible technique and gives rise to a wide family of singular examples. 

In order to understand the examples, we develop a boundary regularity theory for a class of area minimizing $m$-dimensional currents whose boundary consists of smooth $(m-1)$-dimensional surfaces with multiplicities meeting along an $(m-2)$-dimensional smooth submanifold.
\end{abstract}
\maketitle
\tableofcontents
\section{Introduction}

This is the final paper in a series of three —\cite{ian2024uniqueness}, \cite{ianreinaldo2024regularity}, and the present article — providing a \textit{sharp} answer to a long-standing open question posed in Allard’s 1969 PhD thesis \cite{AllPhD} on the boundary regularity of area-minimizing $m$-currents in arbitrary dimension, codimension, and boundary multiplicity, under a convexity assumption. Our results sharply extend Allard’s celebrated 1975 boundary regularity theorem \cite{AllB} to the much more delicate and previously unexplored case of higher boundary multiplicity in the minimizing setting. Before our work, even the existence of a single regular boundary point was not known in this context. 

In this paper, we construct a \textit{family} of sharp examples that demonstrates the dimensional sharpness of the regularity theory we obtain in \cite{ianreinaldo2024regularity}. The construction relies on considering certain area-minimizing currents whose boundary has nontrivial topology. For topological reasons, this forces an essential singularity at the boundary of such an area-minimizing current. This makes the methods presented in this paper for constructing essential boundary singularities extremely flexible, and notably, they do not rely on constructing a calibration.

As part of the study of these examples, we need to prove an excess decay-type theorem at points of ''corner type''. The boundaries considered here consist of smooth surfaces with multiplicity meeting along a smooth $m-2$ dimensional common interface and we get decay to the tangent cone at the points on this interface. Under a suitable ''multiplicity $1$'' assumption, every corner point is regular.

The methods behind the construction of the examples go \textit{beyond} demonstrating the sharpness of the boundary regularity theory: they highlight a conceptually new and flexible mechanism which combines topological constraints with structural regularity results, producing a family of area-minimizing currents with essential boundary singularities.

The problem of understanding regularity theory with higher multiplicity boundaries, in the precise form we study here, was first raised in Allard’s 1969 PhD thesis \cite{AllPhD}, and has remained open ever since. It was later highlighted in White’s famous list of open problems from the 1984 AMS Summer Institute in Geometric Measure Theory \cite[Problem~4.19]{GMT_prob}, and more recently in the ICM 2022 survey by De Lellis \cite{de2022regularity}.

For a short survey, which announces the results in this article together with \cite{ian2024uniqueness} and \cite{ianreinaldo2024regularity}, we refer the reader to the proceeding article \cite{ian2024allard}.

We will give a statement of the example provided in this paper. In order to do this, we must first define the notion of essential singularity.

\begin{definition}\label{d:essential}Let $T, \Gamma$ in $\bB_1 \subset \RR^{m+n}$ such that $T$ is an $m$-dimensional area-minimizing integral current, $\Gamma$ an $m-1$-dimensional submanifold of $\RR^{m+n}$ with $\partial T \res \bB_1= Q\a{\Gamma}$. We say that $0$ is an \textit{essential one-sided singularity} if $\Theta(T,0)=Q/2$ and there is no radius $\rho>0$ such that
\begin{equation}\label{eq:removable}
T \res \bB_{\rho}=\sum Q_i \a{\Sigma_i}
\end{equation}
where $\Sigma_i$ are smooth minimal surfaces in $\bB_{\rho}$ and $\partial \Sigma_i=\Gamma \cap \bB_{\rho}$ (but they are not required to be disjoint outside of the boundary or meet transversal at the boundary). 
\end{definition}

\begin{resultado}\label{t:exampleintro}
There exist a $3$-dimensional area-minimizing current $T$ in $\bB_1 \subset \mathbb{R}^5$ and a real analytic surface  $\Gamma$ such that $\partial T = 2\a{\Gamma}$, 
$\Theta(T,p) = 1$ for every $p \in \Gamma \cap \bB_1 $, and $0$ is an essential one-sided boundary singularity.
\end{resultado}

\subsection{Historical overview} 
The construction of singular examples of area-minimizing currents presents substantial challenges, typically requiring the construction of a calibration.

In codimension $1$, the proof of the minimality of the Simons cone in 1969 \cite{BDG} was a turning point in the history of regularity of area-minimizing currents since it provided the first example of a singular area minimizing integral current in codimension $1$. The proof was later simplified in \cite{de2009short}. In higher codimension, a large family of examples is given by complex varieties due to Federer \cite{Federer1965}, since they are automatically calibrated and thus area minimizing.

It is worth mentioning that recently Liu, in a series of papers, constructed examples of interior singularities of area-minimizing currents, of fractal type, of non-smoothable type, and examples of non-calibtrated area-minimizing currents,  in \cite{liu2021conjecture,liu2022homologically,liu2023homologicallyareaminimizingsurfacescalibrated}.

In the context of boundary regularity for area-minimizing currents, White in 1997 showed that in a parametric context there are no true boundary branch points for a 2-dimensional area-minimizing current with a real analytic boundary curve \cite{White97}. A true branch point is a critical point of the parameterization that is also a singular point for the area-minimizing current induced by the map. In that work, White posed the question of whether an area-minimizing current with a real analytic boundary has finite topology, which is equivalent to finiteness of the interior and boundary singular sets and to the area-minimizing current being a Douglas - Radò solution. In contrast, for a smooth boundary, Gulliver in 1991 \cite{Gulliver} provided an example of a true boundary branch point for a minimal surface (not minimizing) in $\RR^3$.

In \cite{Jonas2}, Hirsch shows an example of a holomorphic function vanishing to infinite order at the boundary in the half-space. This example is used in \cite{DDHM} to show that there is a two dimensional area-minimizing current whose boundary is a smooth (but not real analytic) curve whose boundary singular set has Hausdorff dimension $1$. Furthermore, in \cite{de2022nonclassical}, De Lellis, De Philipis and Hirsch provide an example of an area-minimizing current in a Riemannian manifold whose boundary is a smooth curve and the associated surface has infinite topology. These examples show that the behaviour when the boundary is smooth but not real analytic can be quite singular - unlike the regularity predicted by White's conjecture in the real analytic case.

\subsection{Brief description of the example} Theorem \ref{t:exampleintro} demonstrates the dimensional sharpness of the regularity theory developed with Reinaldo Resende in \cite{ianreinaldo2024regularity}. The example contains an essential boundary singularity and provides a counterexample to full Allard-type boundary regularity for a 3-dimensional area-minimizing current with boundary multiplicity 2. The essential singularities are analogous to the true branch points considered by White and Gulliver, but arise in the higher multiplicity setting of density 1, multiplicity 2 boundary points.

The example comes from a more general theorem, which roughly speaking says that if for $T$ an area-minimizing current and $\partial T$ looks like Figure \ref{img:boundaryexample}, then it has an essential boundary singularity. The defining features of the picture are that $\partial T$ consists of two pieces, one which is topologically a disk $\Gamma^0$ (taken with multiplicity $2$) and the other which is a Möbius strip $\Gamma^1$ which contains $\partial \Gamma^0$ and twists around it. There are additional, technical requirements which can be seen in Theorem \ref{t:examplegeneral}.

The theorem is a consequence of the Hölder continuity of the multivalued normal map, which we show to hold on the spherical cap in the picture up to its boundary given by the black wire. This relies on the multivalued Hölder estimate at smooth boundary points established in \cite{ian2024uniqueness} and employs, at the corner points (the black wire), Simon's method for the uniqueness of tangent cones. The Hölder continuity of the multivalued normal map in the spherical cap up to the boundary forces the map to agree at the black wire to the normal map induced by the Möbius strip. The non-trivial topology induced by the boundary data of this multivalued map generates an essential boundary singularity for the spherical cap away from the black wire.

\begin{figure}[H] 
    \centering 
    \includegraphics[width=0.7\linewidth]{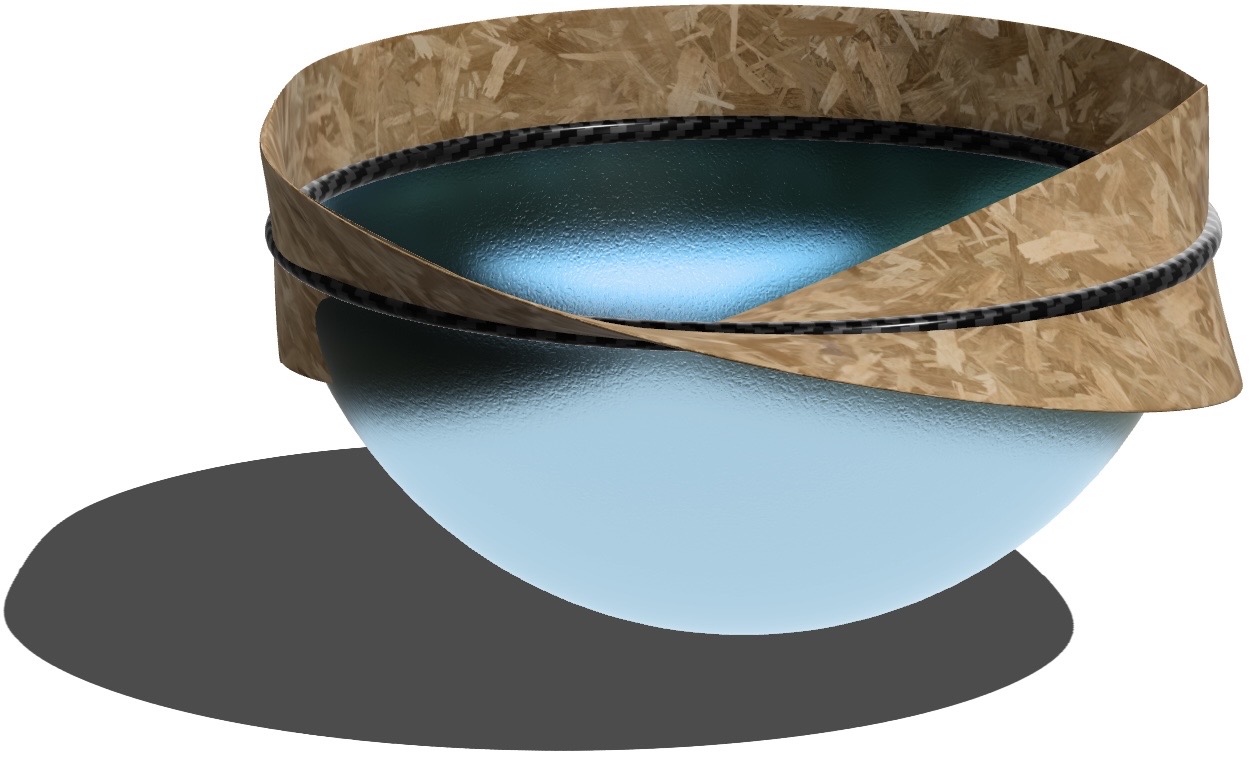}
    \caption{\small Boundary of example}
    \vspace{-10pt}
    \label{img:boundaryexample}
\end{figure}

See the following sketch of the proof for a more detailed discussion of the argument in the simplified setting of the linear problem.

\subsection{Sketch of the proof}
We explain the procedure we use to construct the example in the particular case of the linear problem. The linear problem consists of multi-valued functions $u$ which vanish at the boundary of a certain domain. We will work with multivalued functions, understood as in \cite{Alm} and \cite{DS1}.

\begin{assumption}[Linear Problem]
Given a Dir-minimizing function $u \in W^{1,2}(\Omega,A_{Q}(\RR^n))$ for an open set $\Omega \subset \RR^m$. The boundary linear problem consists of functions $u$ such that $u \res \partial \Omega \cap \bB_1=Q\a{0}$.
\end{assumption}

In the interior setting, a simple example of an essential singularity is the two-valued square root
$u \in W^{1,2}(\RR^2, \mathcal{A}_{2}(\RR^2))$ 

\begin{equation*}
u(z)=\sum_{w^2=z}\a{w}
\end{equation*}
where we identify $\mathbb{C}$ with $\RR^2$.

A natural singular candidate for the boundary linear problem that we would be interested whether it is energy minimizing is $f \in W^{1,2}(\RR^2 \times \RR^{+},\mathcal{A}_2(\RR^2))$ defined by  
\begin{equation*}
    f(z,t)=\sum_{w^2=z}\a{wt}
\end{equation*}
where we identify $\RR^2$ with $\mathbb{C}$ both in the domain and in the image.

This function can be shown to be stationary and it is unclear whether it is energy minimizing. Moreover, even if one were able to demonstrate, in a complicated way, that $f$ is energy minimizing, there seems to be no reason to expect that such a strategy would translate directly into the construction of boundary singularities for an area-minimizing current of the desired type. 

\textbf{Construction of the singular Dir-minimizer}

We consider the cylinder $\mathbf{C}:=\left\{z \in \RR^2: |z| \leq 1\right\} \times [0,1]$. We can just take $u$ a solution for the linear problem with boundary data $f$ as above. This is
\begin{equation*}
u \in W^{1,2}(\mathbf{C},\mathcal{A}_2(\RR^2)) \; \textup{Dir-minimizing with} \; u \res \partial \mathbf{C} =f \res \partial \mathbf{C}.
\end{equation*}

We might have $u=f$ if the function $f$ were energy minimizing. We can show that regardless of what $u$ is, it will have an essential boundary singularity at some non-corner boundary point (i.e., in $\left\{ |z|<1\right\} \times \left\{0\right\}$).

As we mentioned earlier, showing that $u=f$ or that $f$ is energy minimizing is a technically challenging question which would likely (if true) have a very rigid proof. On the other hand, the proof we provide of the singular behaviour is very flexible.

\textbf{The normal derivative as a multivalued map}

By Theorem 10.3 of \cite{ian2024uniqueness}, we know that there is a unique blowup for the linear problem.
This implies that there exists  a $2$-valued "normal derivative" map that is well defined at the boundary $\left\{|z|<1,t=0\right\}$. The map $\eta: \left\{|z|<1, t=0\right\} \rightarrow \mathcal{A}_2(\RR^2)$ turns out to locally be a Hölder continuous map in $\left\{|z|<1,t=0\right\}$, by the theory developed in \cite{ian2024uniqueness}.

The theory that we develop in this paper allows us to extend the Hölder continuity to the boundary of the domain where the multivalued normal derivative is defined. We show an excess decay type theorem at points in corners, which for the linear setting corresponds to understanding the rate of decay of $u$ at $\mathbb{S}^1 \times \left \{0\right\}$. Moreover, those points can shown to be regular points. This is, the function $u$ decouples in a neighborhood of $\mathbb{S}^1 \times \left \{0\right\}$ as two classical harmonic functions whose graphs do not intersect in the interior.

This allows us to extend the definition of $\eta$ to $\left \{|z| \leq 1 ,t=0\right\}$ in a Hölder continuous way. 

By definition, $u$ agrees with $f$ in $\partial C$. Furthermore, since $u$ is regular at $\mathbb{S}^1\times \left\{0\right\}$
\begin{equation*}
\eta(z)= \sum_{w^2=z} \a{w}.
\end{equation*}

\textbf{Topological existence of an essential singularity}
We concluded that $\eta$ is a continuous $2$-valued function on the unit disk $\left\{|z| \leq 1 ,t=0\right\}$, which at the boundary  $\left \{|z|=1 ,t=0\right\}$ must take the complex square root.  This implies that $\eta$ must have at least one essential singularity $p \in \left \{|z|<1 ,t=0\right\}$. The argument goes as follows: if all the points where either regular or non-essentially singular, then in a neighborhood of every point $\eta$ splits as two smooth functions. Since $\left \{|z| \leq 1, t=0\right\}$ is simply connected, we must thus have that $\eta$ splits globally on $\left \{ |z| \leq 1, t=0\right\}$ as two smooth graphs. This gives a contradiction to the fact that $\eta(\left\{|z|=1,t=0 \right\})$, as a subset of $\RR^5$, is a smooth connected curve with no-self intersections which is taken with multiplicity $1$. See Proposition \ref{p:essentialsingularity}.

A point $p$ at which $\eta$ has an essential singularity must be an essential boundary singularity for $u$. Indeed, if $p$ was not an essential singularity for $u$ then there exists $\rho$ such that
\begin{equation*}
    u \res \bB_{\rho}(p)=\a{u_1}+\a{u_2}
\end{equation*}
where $u_1$ and $u_2$ are two classical harmonic functions. This would imply that the normal derivative at the flat piece of the boundary of $\mathbf{C}$ (i.e., $\left \{(z,t) \in \mathbb{C} \times \RR:|z| \leq 1, t=0 \right\}$) can be written as
\begin{equation*}
\eta \res \bB_{\rho}(p)= \a{\frac{\partial}{\partial t}|_{t=0}u_1}+\a{\frac{\partial}{\partial t}|_{t=0}u_2}.
\end{equation*}
This contradicts the fact that $p$ is an essential singularity for $\eta$.

As we have seen, we only require topological information and this construction does not rely on a calibration. Indeed, these methods work on any sufficiently smooth boundary data which induces non-trivial topology.

\subsection{Higher multiplicity Allard boundary regularity}
We introduce the general type of assumption we considered in \cite{ian2024uniqueness} and \cite{ianreinaldo2024regularity}.
\begin{assumption}[Assumption $C^{k,\alpha}$]\label{A:general}
Let integers $m, n\geq 2$, $Q\geq 1$, $k \in \mathbb{Z}_{\geq 0} \cup \left\{\infty, \omega \right\}$ and $\alpha \in [0,1]$. Consider $\Gamma$ a $C^{k,\alpha}$ oriented $(m-1)$-submanifold of $\RR^{m+n}$ without boundary, and assume $0\in\Gamma$ (when $k=\omega$, $\Gamma$ is a real-analytic submanifold of $\RR^{m+n}$). Let $T$ be an integral $m$-dimensional area-minimizing current in $\bB_1$ with boundary $\partial T \res \bB_1=Q\a{\Gamma \cap \bB_1}$.
\end{assumption}
The convex barrier assumption is the following:
\begin{assumption}[Convex Barrier] \label{convexbarrier} Let $\Omega \subset \RR^{m+n}$ be a domain such that $\partial \Omega$ is a $C^2$ uniformly convex submanifold of $\RR^{m+n}$. We say that $\sum_{i=1}^N Q_i \a{\Gamma_i}$ has a convex barrier if $Q_i$ are positive integers and $\Gamma_i \subset \partial \Omega$ are disjoint $C^2$ closed oriented submanifolds of $\partial \Omega$. In this setting we consider $T$ an area-minimizing current with $\partial T=\sum_i Q_i \a{\Gamma_i}.$ 
\end{assumption}

Theorem \ref{t:exampleintro} shows that our regularity theory is dimensionally sharp and moreover that in this setting (i.e., density $Q/2$) the situation is not dimensionally better if we wish to study only essential singularities. Moreover, the situation does not get better with the assumption of a real analytic boundary. The theorem we obtained in \cite{ianreinaldo2024regularity} is the following:

\begin{theorem}[Rectifiability of the one-sided singular set \cite{ianreinaldo2024regularity}]\label{t:retifiability} Let $T$ be an area-minimizing current which satisfies Assumption \ref{A:general} with $k=3$, $\alpha>0$ and $m \geq 3$. Further assume that $\Theta(T,p)=Q/2$ for every $p \in \bB_1 \cap \Gamma$. The boundary singular set (the set of points on $\bB_1 \cap \Gamma$ which are not one-sided regular points) is $\HH^{m-3}$ rectifiable (countable if $m=3$). Furthermore, if $\partial T=\sum_{i=1}^N Q_i \a{\Gamma_i}$ has a convex barrier (i.e. it satisfies Assumption \ref{convexbarrier}), the boundary singular set is $\HH^{m-3}$ rectifiable (countable if $m=3$). 
\end{theorem}

In the context of $Q=1$ and dimension $m=2$, based on known examples of \cite{DDHM}, the situation with essential boundary singularities seems to be different with that of general boundary singularities. The general boundary singular set can be of Hausdorff dimension $1$, but those singularities are of removable type, in their context it means there exists $\rho>0$ such that
\begin{equation*}
T \res \bB_{\rho}= \a{\Sigma_1} +\a{\Sigma_2}
\end{equation*}
with $\Sigma_1$ and $\Sigma_2$ two smooth minimal surfaces with $\partial \a{\Sigma_1}=\a{\Gamma}$ and $\partial \a{\Sigma_2}=0.$ There is no known example in the literature of a smooth embedded loop $\Gamma \subset \RR^{2+n}$ and a $2$-dimensional area-minimizing current $T$ with $\partial T=\a{\Gamma}$ such that $T$ has infinitely many essential boundary singularities.

\subsection{Main theorems of this work}

In order to construct the example we develop a regularity theory for currents whose boundaries have corners.

\begin{assumption}\label{a:corneredboundarymonotonicity} Consider the following:
\begin{enumerate}[(i)]
\item Let $j \in \left\{0,1\right\}$, $N^j$ be two positive integers. Let $\Sigma$ be a $C^2$ ambient manifold in $\RR^{m+n}$.
\item Let $\Gamma_i^j$, $1 \leq i \leq N^j$ be $m-1$ dimensional $C^2$ submanifolds of $\Sigma$ such that $\partial \a{\Gamma_i^j}=(-1)^j\a{M}$ for $M$ a $m-2$ dimensional $C^2$ submanifold of $\Sigma$ with $0 \in M$. We further assume that $\Gamma_i^j$ are transversal to each other at $M$. We also denote $\mathbf{\Gamma}:=\cup_{i,j} \Gamma_i^j.$
\item Let $Q_i^j$ be positive integer multiplicities with  $\sum (-1)^jQ_i^j =0.$  We also denote $Q=\sum_i Q_i^0=\sum_i Q_i^1.$
\item Every point in $M$ is a good corner point, as defined in later in Definition \ref{d:goodcorner}. This is a technical condition that ensures a bound on the mass.
\end{enumerate}
The current $T$ is area minimizing in $\bB_1 \cap \Sigma$  with cornered boundaries $\Gamma_i^j$ if 
\begin{equation*}
    \partial T=\sum_{0 \leq j \leq 1} \sum_{1 \leq i \leq N^j} Q_i^j \a{\Gamma_i^j}.
\end{equation*}
\end{assumption}

The boundary pieces $\Gamma_i^0$ are those with positive orientation with respect to the orientation of $M$. The boundary pieces $\Gamma_i^1$ are those with negative orientation with respect to $M$.

We provide an example of a current taking the boundary data.
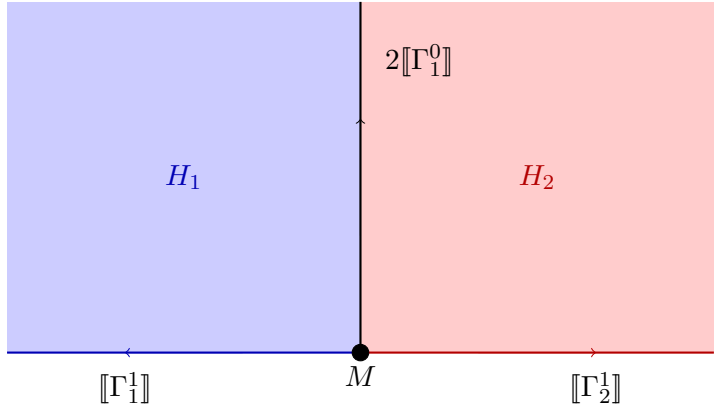
\begin{figure}[H]
    \centering
    \begin{minipage}{0.35\textwidth}  
        \raggedright  
        
        We consider $T=H_1+H_2$, where $H_1$ and $H_2$ are quarter pieces of a plane which are oriented such that:  
        \begin{equation*}
        \partial [\![ H_1 ]\!] = [\![\Gamma_1^0]\!]  - [\![ \Gamma_1^1 ]\!]
        \end{equation*}
        \begin{equation*}
        \partial [\![ H_2 ]\!] = [\![\Gamma_1^0]\!] - [\![ \Gamma_2^1 ]\!]
        \end{equation*}
        The boundary pieces $\Gamma_0^1$, $\Gamma_1^1$, \\ and $\Gamma_2^1$ are half-planes.\\
        $N^0=1,N^1=2$, $Q_1^0=2$, $Q_1^1=Q_2^1=1$.    \end{minipage}%
    \hfill  
    \begin{minipage}{0.65\textwidth}  
        \begin{tikzpicture}[scale=1.55] 
            \fill[fill=blue!20] (-3,0) rectangle (0,3); 
            \fill[fill=red!20] (0,0) rectangle (3,3); 
            \draw[thick, blue!70!black] (-3,0) -- (0,0); 
            \draw[thick, red!70!black] (0,0) -- (3,0); 
            \draw[thick, black] (0,0) -- (0,3); 
            \filldraw [black] (0,0) circle (2pt);
            \node[black] at (0, -0.2) {$M$}; 
            \draw[->, red!70!black] (1,0) -- (2,0); 
            \draw[->, blue!70!black] (-1,0) -- (-2,0); 
            \draw[->, black] (0,1) -- (0,2); 
            \node[black] at (0.5, 2.5) {$2[\![\Gamma_1^0]\!]$}; 
            \node[black] at (-2, -0.3) {$[\![ \Gamma_1^1 ]\!]$}; 
            \node[black] at (2, -0.3) {$[\![ \Gamma_2^1 ]\!]$}; 
            \node[blue!70!black] at (-1.5, 1.5) {$H_1$}; 
            \node[red!70!black] at (1.5, 1.5) {$H_2$}; 
        \end{tikzpicture}
    \end{minipage}
    \caption{Example of Corner}
    \label{img:nicecorneredopenbook}
\end{figure}

We simplify the notation of the main theorem for cornered 
boundaries. For the precise statements see Assumption \ref{a:corneredboundary} and  Theorem \ref{t:decomposition}
\begin{resultado}[Excess decay]
\label{r:excessdecay}We consider $T, \Gamma_{i,j},M, \Sigma$ as in Assumption \ref{a:corneredboundarymonotonicity}.
\begin{itemize}
    \item Assume for every $1 \leq i_0 \leq N^0$, $1 \leq i_1 \leq N^1$ we have $\Gamma_{i_0}^0$ and $\Gamma_{i_1}^1$ meet perpendicularly at $M$.
    \item Assume, in addition, that for any pair of surfaces with the same orientation $j=0,1$, $\Gamma_{i_0}^j$ and $\Gamma_{i_1}^j$, meet at $M$ with an angle which is uniformly bounded.
\end{itemize}  Then there is a unique tangent cone at every point of $M$ at which $T$ decays with a power rate.
\end{resultado}

\begin{resultado}[Regularity in the multiplicity $1$ case]  Assume that we are in the hypothesis of Theorem \ref{r:excessdecay}. Additionally assume that there is only one classical boundary piece with positive orientation and all the sheets with negative orientation have multiplicity $1$ (i.e., $N^0=1$, $N^1=N$, $Q_1^0=N$, $Q_1^1=Q_2^1=...=1$).

Then $T$ is regular along $M$: for every $p \in M$ there exists $\rho$ such that 
\begin{equation*}
T \res \bB_{\rho}(p)=\sum_{i=1}^N\a{\Sigma_i}
\end{equation*}
where $\Sigma_i$ are minimal surfaces, which intersect only at $M$ and $\partial \a{\Sigma_i}=\a{\Gamma_0^0}-\a{\Gamma_i^1}$. The minimal surface $\Sigma_i$ is smooth away from $M$ and attaches smoothly to $\Gamma_0^0$ and $\Gamma_i^1$.
\end{resultado}
These assumptions are not restrictive, as the theorems above suffice for our applications. Given that we have a decomposition theorem, the regularity theory with higher multiplicity is equivalent to the case when $N^0=1$, $N^1=1$, $Q^0=Q^1=Q$. It would be reasonable to expect the same type of regularity theory we developed in \cite{ianreinaldo2024regularity} to apply for this setting, using the same methods. Assuming the same boundary and ambient manifold smoothness which allows us to run the regularity theory in \cite{ianreinaldo2024regularity}, one would expect the singular set of $T$ along $M$ to be $\HH^{m-3}$ rectifiable.

As a corollary, we get the following theorem which holds for a flexible family of boundaries
\begin{resultado}\label{t:examplegeneral}
Let $\Omega$ be a $C^2$ uniformly convex open set. 
\begin{itemize}
\item Let $M \subset \partial \Omega$ be a $C^2$ simple closed curve.
\item Let $\Gamma^0$ be a $C^2$ surface with $\partial \a{\Gamma^0}=M$ diffeomorphic to a $2$ dimensional disk.
\item Let $\Gamma^1 \subset \partial \Omega$ be an integral $2$ dimensional current. We assume that $M \subset \textup{spt}(\Gamma^1)$ and in a neighborhood $U$ of $M$, $U \cap \textup{spt}(\Gamma^1)$ is a $C^2$ surface diffeomorphic to a Möbius strip which twists once around $M$. Indeed, $U \cap \textup{spt}(\Gamma^1) \setminus M$ will be connected. We assign multiplicity $1$ to $\textup{spt}(\Gamma^1)$ in $U$. We suppose that $\Gamma^1$ is oriented such that $\partial \a{\Gamma^1}=-2\a{M}$. The surface $U \cap\Gamma^1$ is a smooth surface whose orientation switches across $M$. 
\end{itemize}
Let $T$ be an area-minimizing current with $\partial T=2\a{\Gamma^0}+\a{\Gamma^1}$. Then $T$ is regular in a neighborhood of $M$ and has an essential boundary singularity at some point of $\Gamma^0 \setminus M$. 
\end{resultado}

We show that, at the disk $\Gamma^0$, we can define a $2$-valued normal map, induced by the unique tangent cone to $T$. The normal map, defined at $\Gamma^0$, is Hölder all the way to the boundary $M$. At the boundary $M$, the normal map is given by the tangent to $\Gamma^1$ at $M$, which is topologically non-trivial since $\textup{spt}(\Gamma^1)$ is a Möbius strip in a neighborhood of $M$. The topological non-triviality implies, as we saw in the linear problem, that there is an essential singularity at some interior point of $\Gamma^0$.

We note that $\Gamma^0$ and $\Gamma^1$ are required to lie on a convex barrier in order to be able to apply the theorems of \cite{ian2024uniqueness} as well as the corresponding analogues at points of $M$ established in this paper.

\subsection{Acknowledgments} 
The author is extremely grateful to his PhD advisor, Professor Camillo De Lellis, for his outstanding mentorship and extraordinary support throughout this project. The author also wishes to express gratitude to Reinaldo Resende for providing detailed feedback on the manuscript. 
\subsection{Basic definitions}
Throughout the paper, we use the notation '$\lesssim$' to denote inequality up to a multiplicative constant, which may depend only on $m$, $n$, $\overline{n}$, and $Q$. We also use $C(\alpha_1,...,\alpha_k)$ to denote a constant which depends on $\alpha_1,...,\alpha_k$.

We define $\bB_{r}(p) := \{x \in \RR^{m+n} : |x - p| < r\}$. We typically use $V$ for $m-1$ dimensional planes. We use $L := \RR^{m-2} \times \{0\} \subset \RR^{m+n}$ as the spine of our cones. For a plane $\pi$, let $\mathbf{p}_{\pi}$ denote the projection onto $\pi$. For any $\varepsilon > 0$, we use $B_{\varepsilon}(L)$ to represent the $\varepsilon$-neighborhood of $L$, that is, $B_{\varepsilon}(l) = \{x \in \RR^{m+n} : \dist(x, L) < \varepsilon\}$. 

We assume each $k$-dimensional linear subspace \(\pi\) of $\RR^{m+n}$ is oriented by a \(k\)-vector \(\vec{\pi} = v_1 \wedge \cdots \wedge v_k\), where \((v_i)_{i=1}^k\) forms an orthonormal basis for $\pi$. Additionally, we adopt the notation \(\left|\pi_2 - \pi_1\right|\) for \(\left|\vec{\pi}_2 - \vec{\pi}_1\right|\), where \(|\cdot|\) represents the norm tied to the standard inner product of $k$-vectors. The $k$-dimensional Hausdorff measure on $\RR^{m+n}$ is written as $\HH^k$. We denote $\mathcal{A}_Q(\RR^{m+n})$  the $Q$-tuples in $\RR^{m+n}$ for some integer $Q \geq 1$. We denote by $\mathcal{G}(p_1,p_2)$, the Wasserstein distance between a pair of points $p_1, p_2$ on the space $\mathcal{A}_Q(\RR^{m+n})$. For more on multi-valued functions, see \cite{DS1}. We denote by $\mathbf{A}_M$ the $C^0$-norm of the second fundamental form of a smooth submanifold $M$ of $\RR^{m+n}$. We also denote under Assumption \ref{a:corneredboundarymonotonicity}, $\mathbf{A_{\Gamma}}:=\max_{i,j}\mathbf{A}_{\Gamma_i^j}$. 

For further basic definitions and standard notations, we refer to \cite{Fed} and the survey \cite{delellis2015size}.

We will use $\mathbf{E}$ to denote the $L^2$ excess, and local versions which we introduce later. The $L^2$ excess relative to an open book $C$ or cornered open book is defined as
\begin{equation}
    \mathbf{E}(T, C, \bB_r(p)) := \frac{1}{r^{m+2}} \int_{\bB_r(p)} \dist(x, C)^2 d|\!|T|\!|(x).
\end{equation}
We will use $\mathbb{E}$ for a suitable measure-theoretic excess, which will be defined later in equation \eqref{e:strongexcess}. Finally, we introduce the circular projection $\pi_{\circ} : \RR^{m+n} \to \RR^{m}$, defined as    $\pi_{\circ}(x, y) = (x, |y|)$ for $(x, y) \in \RR^{m-1} \times \RR^{n+1}$.
\begin{definition}\label{D:AreaMin}
An $m$-dimensional integral current $T$ in $U$ with $\textup{spt}(T)\subset\Sigma$ is \emph{area minimizing in $\Sigma\cap U$}, if it holds that $\mathbf{M}(T) \leq \mathbf{M}(T+\partial \tilde{T})$, for all $\tilde{T}$ $(m+1)$-dimensional integral current in $U$ with $\textup{spt}(\tilde{T})\subset\Sigma\cap U$. 
\end{definition}
\subsubsection{Cones}
We recall the definition of an open book and introduce their angle:
\begin{definition}
A cone $C$ is an \textit{open book} if it can be written as
 $C=\sum_{i=1}^N Q_i\a{H_i}$ where $H_i$ are distinct half-planes with  $\partial\a{H_i}=\a{V}$
where $V$ is an $m-1$ dimensional plane, referred to as the spine of the cone. We define the minimum angle of an open book as
\begin{equation*}
\alpha(C):=\min_{i \neq j}
\langle H_i,H_j \rangle
\end{equation*}
where $\alpha(C):=1$ if $C=Q\a{H}$ for $H$ a half-plane.
\end{definition}

\begin{definition}
A cone $C$ is a \textit{two-sided flat cone} if 
\begin{equation*}
 C=(Q/2+Q^*)\a{H}-(Q^*)\a{-H} 
\end{equation*} for $Q$ and $Q^*$ integers and $H$ a half plane with $\partial H=V$ where $-H$ denotes the reflection of $H$ along $V$.
\end{definition}
We introduce a distance between open books with the same spine. 
\begin{definition}[Distance Open Books]
Given a pair of open books $C^j$ for $j=0,1$ which are of the form
\begin{equation*}
C^j:=\sum_i Q_i^j\a{H_i^j}
\end{equation*}
 where $\partial \a{H_i^j}=\a{V}$. We take $\eta_i^j$ to be the unit normal to $V$ that determines $H_i$.
We define $\eta^j \in \mathcal{A}_Q(V^{\perp})$ as $\eta_i^j= \sum_i Q_i\a{\eta_i^j}$.
We define 
\begin{equation}\label{eq:distopenbooks}
\mathcal{G}(C^1,C^2):=\mathcal{G}(\eta^1,\eta^2).
\end{equation}
\end{definition}
\begin{definition}\label{d:corneredcone}
A cornered area-minimizing cone $C$ with spine $L$ (of dimension $m-2$) is one that satisfies 
\begin{equation}
\partial C=\sum_{j=0}^1\sum_{1 \leq i \leq N^j} Q_i^j \a{V_i^j}
\end{equation}
where $V_i^j$ are $m-1$ dimensional half-planes with $\partial\a{V_i^j}=(-1)^j\a{L}$ and $Q_i^j$ are positive integers with $\sum_{i}Q_i^0=\sum_iQ_i^1$. 
\end{definition}
\begin{definition}[Quadrant]
We define $H$ to be an $m$-\textit{quadrant} (or just quadrant if the dimension is easily understood) if $H \subset \RR^{m+n}$ is an orthonormal transformation of the set $\left\{(x,y) \in \RR^2: x\geq 0, y \geq 0 \right\} \times \RR^{m-2}$.
We define $\hat{H}$ to be the full $m$-dimensional plane containing $H$.
\end{definition}
\begin{definition}\label{d:corneredopenbook} A cornered open book with spine $L$ (of dimension $m-2$) is 
\begin{equation*}
C=\sum_{1 \leq i \leq N} Q_i\a{H_i}
\end{equation*}
where $H_i$ is a quadrant with $\partial \a{H_i}=\a{V_i^0}-\a{V_i^1}$ where $V_i^0$ and $V_i^1$ are orthogonal half-planes of dimension $m-1$ with $\partial \a{V_i^j}=(-1)^j\a{L}$.  We denote $Q=\sum_{1 \leq i \leq N} Q_i$. 
\end{definition}
\begin{remark}
Cornered open books as we consider will automatically have, if $C$ is centered at $0$, $\Theta(C,0)=Q/4$. This setting is appropriate for the desired applications. More general settings could be considered, such as not requiring $H_i$ to be quadrants and just be pieces of planes.
\end{remark}

The simplest cornered open book is a single quadrant with multiplicity. The main example for the construction of the example later is the one given by Figure \ref{img:nicecorneredopenbook}.

We define $\mathcal{C}$ as the set of cornered open books which have the right boundary data. 
\begin{definition}\label{d:corneredopenbookboundarydata}Let $\Gamma_i^j$, $M$ as in \ref{a:corneredboundarymonotonicity}. We define $\mathcal{C}:M \rightarrow \left\{ \textup{Cornered open books}\right\}$. For each $p \in M$ we will denote $\mathcal{C}_p$ the set of generalized cornered open books with the following conditions.

Given the cornered open book
\begin{equation*}
    C=\sum_{1 \leq i \leq N} Q'_i\a{H_i}
\end{equation*}
it belongs to $\mathcal{C}_p$
if

\begin{itemize}
\item $L=T_{p}(M)$.
\item 
There exists an index selection $k(,0): \left\{1,...,N\right\}  \rightarrow \left\{1,...,N^0\right\}$, $k(,1): \left\{1,...,N\right\}  \rightarrow \left\{1,...,N^1\right\}$.

\item 
\begin{equation*}
V_i^j=T_{p}\left(\Gamma_{k(i,j)}^j \right).
\end{equation*}
\item  
For every $j \in \left\{0,1 \right\}$, $1 \leq k \leq N^j$
\begin{equation*}
Q_k^j=\sum_{k(i,j)=k}Q'_i.
\end{equation*}
\end{itemize}
\end{definition}

\subsubsection{Tangent cones} For $x\in\RR^{m+n}$ and $r>0$, define $\iota_{x, r}(y):= \frac{y-x}{r}$. For any integral current $T$, we define the rescaled current $T$ at $x$ at scale $r$ as ${\iota_{x, r}}_{\sharp} T=:T_{x,r}$ and $T_r =: {\iota_{0,r}}_{\sharp} T$. The current $T_{x}$ is called a blow-up of $T$ at $x$, if there exists a sequence of radii $r_j\to0$ such that $T_{x,r_j}\to T_x$ in the weak topology. By  \cite{Theodora}, it follows from the monotonicity formula that if $T$ is area minimizing in $\Sigma$ and $\partial T = Q\a{\Gamma}$ for $\Gamma\in C^{1,\kappa}$, every convergent subsequence of $T_{p,r_j}$ converges to an oriented cone about $p$ for any $p \in \Gamma$ and any $r_j\downarrow 0$. An oriented cone about $p$ is understood as a current $C$ such that $C_{p,r} = C$ for any $r>0$. See Theorem \ref{t:monotonicityformula} for the precise statements of the monotonicity formula and first variation formula for an area-minimizing current $T$ in $\Sigma$ with $\partial T=Q\a{\Gamma}$ where $\Sigma$ and $\Gamma$ are $C^2$ submanifolds. In this paper, we will also need to consider a type boundaries which consist of smooth $(m-1)$-dimensional submanifolds meeting at a smooth $(m-2)$-dimensional smooth submanifold. In Theorem \ref{t:firstvariationcorneredboundary} we show the first variation formula to hold  at the $(m-2)$-dimensional smooth submanifold, and a monotonicity formula with error terms Corollary \ref{c:monotonerrorcorner}. This implies existence of tangent cones at corner points.

\part{The study of corner points}
\section{Monotonicity formula at corner points}

In this section, we show the validity of the first variation formula in this context and the monotononicity formula with error terms under appropriate assumptions.

\subsection {Smooth boundary setup}
We introduce the setup for the monotonicity formula used in \cite{ian2024uniqueness}.

\begin{assumption}\label{A:generalsingleboundary}
Let $m,n\geq 2$, $\bar{n}\geq 1$, $Q\geq 1$ be integers. Consider $\Sigma$ an $(m+\bar{n})$-manifold of class $C^2$ in $\RR^{m+n}$, $\Gamma$ a $C^2$ oriented $(m-1)$-submanifold of $\Sigma$ without boundary, and assume $0\in\Gamma$. Let $T$ be an integral $m$-dimensional area-minimizing current in $\Sigma\cap\bB_1$ with boundary $\partial T \res \bB_1=Q\a{\Gamma \cap \bB_1}$.
\end{assumption}

Given $T, \Gamma, \Sigma$ as in \ref{A:generalsingleboundary}.
We define
\begin{equation*}
\Theta_{i}(T,p,r):=\textup{exp}(C_0\mathbf{A}_{\Sigma}^2r^2)\frac{|\!|T|\!|(\bB_r(p))}{\omega_mr^m}
\end{equation*}
and 
\begin{equation*}
\Theta_{b}(T,p,r):=\textup{exp}(C_0\left(\mathbf{A}_{\Sigma}+\mathbf{A_{\Gamma}}\right)r)\frac{|\!|T|\!|(\bB_r(p))}{\omega_mr^m}
\end{equation*}
where $\mathbf{A}_{\Sigma}$ and $\mathbf{A_{\Gamma}}$ denote the $C^0$ norm of the second fundamental form.

\begin{theorem}[Monotonicity formula - Theorem 2.4 \cite{ian2024uniqueness}]\label{t:monotonicityformula}
Let $T, \Gamma, \Sigma$ be as in Assumption \ref{A:general}. 
\begin{itemize}
\item If $p \in \bB_1 \setminus \Gamma$ then $r \rightarrow \Theta_{i}(T,p,r)$ is monotone on the interval 

$(0,\min\left\{\dist(p,\Gamma),1-|p|\right\})$.
\item If $p \in \bB_1 \cap \Gamma$, then $r \rightarrow \Theta_{b}(T,p,r)$ is monotone on $(0,1-|p|)$.
\end{itemize}
Thus the density exists at every point. Moreover, the restrictions of the map $p \rightarrow \Theta(T,p)$ to $\bB_1 \cap \Gamma$ and $\bB_1 \setminus \Gamma$ are upper semi-continuous. 

Additionally, this implies that blowup limits are cones, both at the interior and at the boundary.

If $X \in C^1_{c}(\bB_1,\RR^n)$, then
\begin{equation}\label{eq:firstvariationclassicalboundary}
\delta T(X)=- \int_{\bB_1} X \cdot \overrightarrow{H}_{T}(x)d|\!|T|\!|(x)+Q \int_{\Gamma \cap \bB_1} X \cdot \overrightarrow {\eta}(x) d\HH^{m-1}(x)
\end{equation}
where $\overrightarrow{H}_{T}=\sum_{i=1}^m A_{\Sigma}(e_i,e_i)$ where $\overrightarrow{T}=\wedge_{i=1}^m e_i$ is an orthonormal basis of $\overrightarrow{T}$. The vector $\overrightarrow{\eta}$ is defined $\HH^{m-1}$ almost everywhere as:
\begin{itemize}
\item For every tangent cone $C$ at $p \in \Gamma$ which is an open book $C=\sum_{i=1}^N Q_i\a{H_i}$, 
\begin{equation}\label{eq:vectoropenbook}
\overrightarrow{\eta}=\sum_{i=1}^N \frac{Q_i \overrightarrow{\eta_i}}{Q}
\end{equation}
where $\overrightarrow{\eta_i}$ is the normal to $T_p(\Gamma)$ which determines the half plane $H_i$. The vector $\overrightarrow{\eta}$ is independent of  the open book $C$ at $p$, but the open book might depend on the blowup sequence.
\item For every tangent cone $C=(Q/2+Q^*)\a{H}-(Q^*)\a{-H}$ at $p \in \Gamma$ which is a two-sided flat cone \begin{equation}\label{eq:twosidedflat}
    \overrightarrow{\eta}=\eta_{H}
\end{equation}
where $\eta_H$ is the normal to $T_p(\Gamma)$ which determines the half plane $H$.
\end{itemize}
\end{theorem}
\begin{corollary}[Monotonicity formula with error terms]\label{c:monotonerror}
Under the hypothesis of the above theorem
\begin{align*}
    r^{-m}|\!|T|\!|(\bB_r(p))-s^{-m}|\!|T|\!|(\bB_s(p))= \int_{\bB_r(p) \setminus \bB_s(p)}\frac{|\left(x-p\right)^{\perp}|^2}{|x-p|^{m+2}}d|\!|T|\!|(x) \\
    +\int_{s}^{r} \rho ^{-m-1} \left[\int_{\bB_{\rho}(p)}(x-p)^{\perp} \cdot \overrightarrow{H}_{T}(x)d|\!|T|\!|(x)+ Q \int_{\Gamma \cap \bB_{\rho}(p)}(x-p) \cdot \overrightarrow{\eta}d\HH^{m-1}\right]d\rho
\end{align*}
\end{corollary}\subsection{Cornered boundary}
We introduce the notion of a good corner point, which is responsible for a mass bound at corner points, and allows us to obtain the first variation formula and monotonicity formula.
\begin{definition}[Good corner point]\label{d:goodcorner} Under Assumption \ref{a:corneredboundarymonotonicity}[(i-ii-iii)], a point $q \in M$ is a good corner point if there exists $r_q>0$ such that for every $p \in \bB_{r_q/2}(q) \cap M$, there exists a diffeomorphism $\Phi_p \in C^2(\bB_{r_q}(p),\RR^{m+n})$
such that
\begin{itemize}
\item $\Phi_p(p)=p$ and $D\Phi_p(p)=I$
\item $\Phi_p(\Gamma_i^j) \subset T_p(\Gamma_i^j)$ (abusing notation and considering the full $m-1$ dimensional plane) and $\Phi_p(M)=T_p(M)$.
\item $D\Phi_p$ maps normal vectors to $M$ into normal vectors to $T_p(M)$. $D\Phi_p$ maps normal vectors to $\Gamma_i^j$ into normal vectors to $T_p(\Gamma_i^j)$.
\item We have the following uniform $C^2$ bound \begin{equation*}
\sup_{p \in \bB_{r_q/2}(q)\cap M} \left\|   \Phi_p\right\|_{C^2(\bB_{r_q}(p))}<\infty.
\end{equation*}
\end{itemize}
\end{definition}

The following lemma is an appropriate modification of \cite[Lemma 4.25]{DDHM} and thus we omit its proof.
\begin{lemma} Under Assumption \ref{a:corneredboundarymonotonicity}, we define $d_p(x):=|\Phi_p(x)|$. There exists $r_0>0$ such that the following holds: 
\begin{itemize}
\item For every $p \in \bB_{1/2} \cap M$, $d_p$ is in $C^2(\bB_{r_0} \setminus \left\{0 \right\})$ and for $x \in \bB_{r_0/2}(p) \setminus \left\{0\right\}$ $D^jd(x)=D^j|x-p|+O(|x-p|^{2-j})$ for $j=0,1,2$ with uniform constants on $p$ in $\bB_{1/2}(0).$
\item $\nabla d_p$ is tangent to $M$ at points of $M \cap \bB_{r_0}(p)$ and tangent to $\Gamma_i^j \cap \bB_{r_0}(p)$ at points of $\Gamma_i^j \cap \bB_{r_0}(p)$.
\end{itemize}
\end{lemma}
We will also need sufficient conditions for $M$ to form a good corner.
\begin{lemma}\label{l:diffeo}
Let $N_i$ be $C^2$ surfaces in $\RR^{m+n}$ with a common boundary $0 \in M$. Assume that any pair of surfaces $N_i$ and $N_j$ their normal vectors to $M$ are orthogonal. Then there exists a neighborhood of $M$ such that the analogous properties of Definition \ref{d:goodcorner} hold.
\end{lemma}
\begin{proof}
We will construct a diffeomorphism $\Phi_p$ with the desired properties by using the Whitney extension theorem. 

Since the construction is similar for $p$ in a neighborhood of $M$ with a uniform $C^2$ norm, we assume without loss of generality that $p=0$.

We start defining $\Phi$ on $M$. We consider a $f_M:M \cap \bB_r \rightarrow T_0(M)$ which is a $C^2$ diffeomorphism with $f_M(0)=0$ and $Df_M=Id$.

We wish to extend the diffeomorphism $f_M$ to diffeomorphisms $f_{N_i}$ of $N_i$. We consider extensions of the diffeomorfisms $f_{N_i}:N_i \cap \bB_r \rightarrow T_0(N_i)$ which extend $f_M$ (possibly reducing $r$). The extensions will be $C^2$ and satisfy $Df_{N_i}=Id$. The extensions can be constructed for instance by using the Whitney extension theorem.

We wish to extend the maps $f_{N_i}$ to a map $\Phi$ with the desired properties which is defined on a ball. In order to construct the desired extension we use the Whitney extension theorem. The condition of the orthogonal intersection between $N_i$ and $N_j$ ensures that the first derivatives and second derivatives are compatible with being derivatives of a $C^2$ function in the ball.

In order to make $\Phi$ map normal vectors to $N_i$, to normal vectors to $T_0(N_i)$ it is enough to prescribe, as part of the data to use the Whitney extension theorem, the derivatives of $\Phi$ along normal vectors to $N_i$ appropriately. Since $N_i$ are $C^2$ manifolds, the appropriate normal derivatives can be taken $C^1$ and chosen to satisfy the compatibility conditions. 

Since the $C^2$ norm is bounded in a neighborhood then up to shrinking the radius we obtain that $\Phi$ is a diffeomorphism.
\end{proof}

\begin{theorem}\label{t:firstvariationcorneredboundary}Let $T,\Gamma_{i,j}, M, \Sigma$ be as in Assumption \ref{a:corneredboundarymonotonicity}. If $X \in C^1_{c}(\bB_1,\RR^{m+n})$, then
\begin{equation}\label{eq:firstvariationcorner}
\delta T(X)=- \int_{\bB_1} X \cdot \overrightarrow{H}_{T}(x)d|\!|T|\!|(x)+\sum_{i,j} Q_i^j \int_{\Gamma_i^j \cap \bB_1} X \cdot \overrightarrow {\eta_{i,j}}(x) d\HH^{m-1}(x)
\end{equation}
where $\overrightarrow{H}_{T}=\sum_{i=1}^m A_{\Sigma}(e_i,e_i)$ where $\overrightarrow{T}=\wedge_{i=1}^m e_i$ is an orthonormal basis of $\overrightarrow{T}$. The vectors $\overrightarrow{\eta_{i,j}}$ satisfy $\overrightarrow{\eta_{i,j}} \leq 1$ and are chosen according to equations \eqref{eq:vectoropenbook} and \eqref{eq:twosidedflat}.
\end{theorem}

\begin{corollary}[Monotonicity formula with error terms]\label{c:monotonerrorcorner}
Under the hypothesis of the above theorem for every $p \in M \cap \bB_1$ and $0<s<r<1-|p|$,
\begin{align}\label{eq:monotonerror}
    r^{-m}|\!|T|\!|(\bB_r(p))-s^{-m}|\!|T|\!|(\bB_s(p))= \int_{\bB_r(p) \setminus \bB_s(p)}\frac{|\left(x-p\right)^{\perp}|^2}{|x-p|^{m+2}}d|\!|T|\!|(x) \\
    +\int_{s}^{r} \rho ^{-m-1} \left[\int_{\bB_{\rho}(p)}(x-p)^{\perp} \cdot \overrightarrow{H}_{T}(x)d|\!|T|\!|(x)+ \sum_{i,j} Q_i^j \int_{\Gamma_i^j \cap \bB_{\rho}(p)}(x-p) \cdot \overrightarrow{\eta_{i,j}}d\HH^{m-1}\right]d\rho
\end{align}
As a consequence, for every $p \in M \cap \bB_1$ there exist tangent cones. This means every sequence of radii $r_i \rightarrow 0$ such that $T_{p,r_i} \rightarrow T_p$ in the weak topology, must have $T_p$ is an area-minimizing cone whose boundary is an union of half-planes with multiplicities.
\end{corollary}
\begin{proof}[Proof of Theorem \ref{t:firstvariationcorneredboundary} and Corollary \ref{c:monotonerrorcorner}]
We start by proving Theorem \ref{t:firstvariationcorneredboundary}. We wish to show that for a vector field $X \in C_{c}^1(\bB_1, \RR^{m+n})$ the identity \eqref{eq:firstvariationcorner} holds. 

Since we have the first variation formula at smooth boundaries with multiplicities (Theorem \ref{t:monotonicityformula}, equation \eqref{eq:firstvariationclassicalboundary}), equation \eqref{eq:firstvariationcorner} holds for any $X \in C_{c}^1(\bB_1, \RR^{m+n})$ such that $\textup{spt}(X) \cap M =\varnothing.$ We can reach that conclusion by using local versions of the first variation formula at interior points and smooth boundary points and patching them with a partition of unity.

Given $h>0$, we consider a bump function $\varphi_h$ supported in an $h$-neighborhood of $M$. We get the following
\begin{align*}
\delta T(X \varphi_h)+\delta T(X(1-\varphi_h))=\int_{\bB_1} X \cdot \nabla_{\overrightarrow{T}} \varphi_hd|\!|T|\!|(x) +\int_{\bB_1} \textup{div}_{\overrightarrow{T}}(X)\varphi_h d|\!|T|\!|(x) \\ - \int_{\bB_1}X(1-\varphi_h) \cdot \overrightarrow{H}_{T}(x) d|\!|T|\!|(x)+ \sum_{i,j}Q_i^j \int_{\Gamma_i^j \cap \bB_1} X \cdot \overrightarrow{\eta}_{i,j}(x)(1-\varphi_h)d\HH^{m-1}(x).
\end{align*}
In the former equation, we expanded the divergence of the vector field $X\varphi_h$ and used the first variation identity for the vector field $X(1-\varphi_h)$, for which equation \eqref{eq:firstvariationcorner} holds as $\textup{spt}(X(1-\varphi_h))\cap M=\varnothing.$

In order to conclude equation \eqref{eq:firstvariationcorner} for the vector field $X$, we take a limit of $h$ to zero. The second term is easily seen to converge to zero. The third and fourth converge to the desired quantities of the right hand side of equation \eqref{eq:firstvariationcorner}.

To conclude, it remains to control the first term involving $\nabla_{\overrightarrow{T}} \varphi_h$. It is enough to show that for every $\varepsilon>0$ 
\begin{equation}\label{eq:singularconer}
\lim_{h \rightarrow 0}\frac{1}{h} |\!|T|\!|(\left\{x \in \bB_{1-\varepsilon} :\dist(x,M)<h\right\})=0.
\end{equation}
This requires a control of the mass along the points on $M$.

We can use similar arguments to  \cite[Theorem 4.4]{DDHMAnnouncement} to show that there is a constant $C$ such that for every $p \in M \cap \bB_{1/2}$ the functions 
\begin{equation}\label{eq:weightedmonotoncorner}
s \rightarrow e^{Cs}\frac{|\!|T|\!|(\left\{d_p<s\right\})}{s^m}
\end{equation}
are monotone in some interval $[0,r_0]$.
This follows from using the first variation formula for the vector field $\nabla d_p$, which has the right regularity to use the first variation formula and maps the corner $M$ onto itself and $\Gamma_i^j$ onto themselves.

The monotonicity formula \eqref{eq:weightedmonotoncorner} gives us a uniform mass bound at balls centered at $M$. Thus we can split the sets
\begin{equation*}
\left\{x \in \bB_{1-\varepsilon} :\dist(x,M)<h\right\}
\end{equation*}
into $C\frac{1}{h^{m-2}}$ balls of size comparable with $h$, each of them, for small enough $h$, with mass bounded by $Ch^m$ by the monotonicity formula \eqref{eq:weightedmonotoncorner}.  This implies that for small enough $h$
\begin{equation*}
\frac{1}{h} |\!|T|\!|(\left\{x \in \bB_{1-\varepsilon} :\dist(x,M)<h\right\}) \leq Ch
\end{equation*}
which allows us to conclude \eqref{eq:singularconer}.

The monotonicity formula with error term follows from the first variation formula as in \cite[Theorem 3.2]{DDHM}. 

The existence of tangent cones follows from the usual argument, which we sketch below. If $p\in M$, $r_i \rightarrow 0$ and $T_{p,r_i} \rightarrow T_p$ in the weak topology. We must have   $\partial T_p=\sum_{i,j} Q_i^j\a{V_i^j}$
 where $V_i^j=T_p(\Gamma_i^j)$. For the current $T_p$, we obtain as a consequence of the monotonicity of \eqref{eq:weightedmonotoncorner} that $\Theta(T_p,r)=\lim_{r \rightarrow 0} \Theta(T,p,r)$. In the limit, for the current $T_p$, the contribution from the second fundamental form vanishes, and the boundary term disappears since each $\overrightarrow{\eta}_{i,j}$ becomes a constant vector orthogonal to $V_i^j$.
 
This implies that  
\begin{equation*}\int_{\bB_1(p) }\frac{|x^{\perp}|^2}{|x|^{m+2}}d|\!|T_p|\!|(x)=0
\end{equation*}
which shows that $T_p$ is a cone.
\end{proof}

\section{Classification of certain area-minimizing cornered cones with density $Q/4$}

In this section, we classify area-minimizing cones with density $Q/4$ and cornered boundary. In general, we also obtain a lower on the density of $Q/4$ which ensures a no-holes condition when applying Leon Simon's uniqueness of cylindrical tangent cones strategy \cite{simon1993cylindrical}.

We refer to Definitions \ref{d:corneredcone} and \ref{d:corneredopenbook} for cornered cones and cornered open books, respectively.

\begin{lemma}\label{l:densitybound}
Let $C$ be a cornered area-minimizing cone as in Definition \ref{d:corneredcone}. Assume that for every $i_0,i_1$, 
\begin{equation}\label{eq:perpendicular}
V_{i_0}^0 \perp V_{i_1}^1.
\end{equation}
Then $\Theta(C,0) \geq Q/4$ with equality if and only if $C$ is a cornered open book.
\end{lemma}
We will need the following technical proposition.
\begin{proposition}\label{p:classification2dcorners}
Let $C$ be a $2$ dimensional area-minimizing cornered cone as in Lemma \ref{l:densitybound}. Then $\Theta(C,0) \geq Q/4$ with equality if and only if $C$ is a cornered open book.
\end{proposition}

\begin{proof}[Proof of Proposition \ref{p:classification2dcorners}]
Assume that $\Theta(C,0)$ is the minimum possible among all cones with the same boundary data. We will show that this is $Q/4$ and the equality case only holds when $C$ is a cornered open book.

By interior regularity for area-minimizing currents due to Almgren \cite{Alm} and De Lellis-Spadaro \cite{DS3,DS4,DS5} ,  $C$ being a $2$d area-minimizing cone, must be fully regular in the interior. 

The link of $C$ (i.e., its restriction to the sphere $\mathbb{S}^{n+1}$) must be stationary outside of $\mathbb{S}^{n+1} \cap \left(\bigcup_{i,j} V_i^j\right).$ This implies that $C$ in $\mathbb{S}^{n+1}$ outside of the boundary must consist of geodesic segments. The geodesic segments can either be a full geodesic or a geodesic joining two boundary points of $C$.

We must then have that 
\begin{equation*}
\textup{spt}(C) \subset \cup_{k} P_k
\end{equation*}
where $P_k$ are planes passing through $\left\{0 \right\}$.

By the constancy theorem applied to the planes $P_k$, $C$ can be decomposed as a sum of finitely many $Q_k \a{H_k}$ where $H_k$ are parts of planes and $Q_k>0$ are positive integers. If the orientations and multiplicities are the ones assigned by the constancy theorem we obtain
\begin{equation*}
C=\sum_k Q_k \a{H_k}, \; |\!|C|\!|(\bB_1)=\sum_k Q_k |H_k|(\bB_1).
\end{equation*}

We classify the pieces of planes we obtain in types - interior and 3 different types of planar pieces:
\begin{enumerate}
\item Interior - $H_k$ is a full plane: $\partial\a{H_k}=0$.
\item Piece of Type 1 - $H_k$ is a piece of a plane (an open subset of an $m$-dimensional plane) which respects the orientation: There exist $i_k^0$ and $i_k^1$ such that
$\partial\a{H_k}=\a{V_{i_k^1}^{0}}-\a{V_{i_k^1}^{1}}$.
\item Piece of Type 2 - $H_k$ is a piece of a plane (an open subset of an $m$-dimensional plane) which reverses the orientation: There exist $i_k^0$ and $i_k^1$ such that
$\partial\a{H_k}=-\a{V_{i_k^1}^{0}}+\a{V_{i_k^1}^{1}}$.
\item Piece of Type 3 - $H_k$ is a piece of a plane (an open subset of an $m$-dimensional plane) which connects two pieces of the same orientation: There exist $i_k^0,i_k^1$ and $j_k$ such that
$\partial \a{H_k}=\a{V_{i_k^0}^{j_k}}-\a{V_{i_k^1}^{j_k}}$.
\end{enumerate}
There can't be interior pieces with non-zero multiplicity. Removing such pieces keeps the boundary of $C$ but reduces $\Theta(C,0)$ and this would contradict the minimality of $\Theta(C,0)$.

We argue in what remains by induction. Let $K$ be the number of planar pieces $H_k$ in the decomposition above and $Q$ as usual is the total multiplicity of the boundary. Assume we have shown that for every cone with $K-1$ sheets in the decomposition and any multiplicity $Q$ we must have $\Theta(C,0) \geq Q/4$ and equality on the cornered open book case. Then the same holds for a cone with $K$ sheets. The base case is $K=1$, which is trivial since it is a single quarter plane.

If $C$ has a piece of type $1$ $H_K$, then removing the piece changes $Q$ to $Q-Q_k$ and $K$ to $K-1$. We can conclude this case by the inductive step, since $\Theta(H_K,0) \geq 1/4$ and the equality holds when $H_K$ is a quadrant.

If $C$ has a piece of type $3$, then we can remove it obtaining a cone $C'$ with total multiplicity $Q$ and $K-1$ sheets. This implies that $\Theta(C',0) \geq Q/4$ and thus $\Theta(C,0)>Q/4$.

If $C$ has a piece of type $2$  and no piece of type $3$, then there must be two pieces $H_1$ and $H_2$ with $\partial \a{H_1}+\partial\a{H_2}=0$, but this implies that $\a{H_1}+\a{H_2}$ is area-minimizing without boundary which gives us the same contradiction as in the interior case.\end{proof}

\begin{proof}[Proof of Lemma \ref{l:densitybound}]
In the context of higher multiplicity boundaries, the analogous to this lemma is Theorem 3.1 \cite{ian2024uniqueness}.

We wish to use monotonicity formula without error terms at points of $L$. If $T$ is any area-minimizing current with boundary $\partial T=\sum_{i,j} (-1)^jQ_i^j \a{V_i^j}$ then the monotonicity formula will have no error terms: indeed, if we take in the monotonicity formula with the error term \eqref{c:monotonerrorcorner} for the case $V_i^j=\Gamma_i^j$ then $\overline{n_{i,j}} \perp V_{i,j}$. Thus the boundary term vanishes. We denote for $p \in L$ \begin{equation*}
    \Theta(T,p)=\lim_{r \rightarrow 0} \frac{|\!|T|\!|(\bB_r)}{\omega_mr^m}.
\end{equation*} Since the monotonicity formula holds without error terms at such points, this limit exists and defines the density.

As usual, if $\frac{|\!|T|\!|(\bB_{r})}{\omega_mr^m}=\Theta(T,p)$ for all $r_0>r>0$ (for some $r_0$) and $p \in L$ then $T$ is a cone at $p$. We obtain
\begin{equation*}
\int_{\bB_1(p)}\frac{|x^{\perp}|^2}{|x|^{m+2}}d|\!|T|\!|=0,
\end{equation*}
which implies that $T$ is a cone. This argument was carried out already as part of the proof of Corollary \ref{c:monotonerrorcorner} to show existence of tangent cones, but we include it here for completeness.

Suppose that $\Theta(C,0)<Q/4$. Let $\Theta(C,0)=\Theta<Q/4$.
Since $C$ is a cone then $\Theta(C,p)=\Theta(C,rp)$ for every $p \in L$ such that $p \neq 0$ and $r>0$. Since the density is upper semi-continuous, then $\Theta(C,p) \leq \Theta$ for every $p \in L$. By taking iterative blowups at points of $L$ we obtain an area-minimizing cone $\hat{C}$ which is of the form 
\begin{equation*}
\hat{C}=\RR^{m-2} \times C' 
\end{equation*}
where $C'$ is a $2d$ area-minimizing cone with 
\begin{equation*}
\partial C'=\sum_{j=0}^1\sum_{1 \leq i \leq N^j} Q_i^j \a{\mathbf{p}_{L^{\perp}}(V_i^j)}
\end{equation*}We conclude by the inequality of Proposition \ref{p:classification2dcorners}. 

Now we wish to study the equality case. As a consequence of the following inequality for every $p \in L$ 
\begin{equation*}
\lim_{r \rightarrow \infty} \Theta(C,p,r)= \lim_{r \rightarrow \infty} \frac{ |\!|C|\!|(\bB_r(p))}{\omega_m r^m} \leq \lim_{r \rightarrow \infty}\frac{(r+|p|)^m}{r^m}\frac{|\!|C|\!|(\bB_{r+|p|})}{\omega_m (r+|p|)^m}=Q/4.
\end{equation*}
We must have that the equality in the monotonicity formula holds for every point. Then it must be a cone with respect to every point in $L$.

Thus
\begin{equation*}
C=\a{\RR^{m-2}} \times C'
\end{equation*}
where $C'$ is a $2$d area-minimizing cone in $\RR^{n+2}$ with cornered boundary 

\begin{equation*}
    \partial C'= \sum_{i,j} (-1)^jQ_i^j\a{\mathbf{p}_{L^{\perp}}(V_i^j)}.
\end{equation*}
We conclude from the equality case of Proposition \ref{p:classification2dcorners}.
\end{proof}

\section{Soft uniqueness of tangent cones at corner points}

In this section, we show uniqueness of tangent cones to hold at corner points with soft arguments. This will not be enough for the desired applications. The uniqueness of tangent cones will be made quantitative in the rest of the paper.

We will define a measure theoretic version of excess as in Section 2.1 \cite{ian2024uniqueness}, with respect to cones $C=\sum_{1 \leq i \leq N} \a{H_i}$ with $\partial \a{H_i}=\a{V_i^0}-\a{V_i^1}$ where $V_i^0$ and $V_i^1$ are halfplanes with $\partial V_i^j=(-1)^j\a{L}$ (and $L$ is a common $m-2$ dimensional plane).

The Wasserstein-2 distance between two measures $\mu^1$ and $\mu^2$ is defined as

\begin{equation*}
    W_2(\mu^1,\mu^2)^2 := \inf \left\{ \int d(x,y)^2 \, d\sigma(x,y) : (\pi_1)_{\#}\sigma = \mu^1 \: \text{and} \: (\pi_2)_{\#}\sigma = \mu^2\right\},
\end{equation*}

where $\sigma$ is a transport plan between $\mu^1$ and $\mu^2$, meaning that $\sigma$ is a measure on $\mathbb{R}^{m+n} \times \mathbb{R}^{m+n}$ such that $(\pi_1)_{\#} \sigma = \mu^1$ and $(\pi_2)_{\#} \sigma = \mu^2$.

Given two measures $\mu^1$ and $\mu^2$, we define the following distance between them:

\begin{equation*}
d(\mu^1, \mu^2) := \inf \left\{ 
W_2(\mu_1^1, \mu_1^2)^2 
+ \sum_{l=1}^2 \int \textup{dist}\left(x,\bigcup_{i,j}V_{i,j}\right)^2 \, d|\mu_2^l|  
\ \middle|\ \mu_1^l + \mu_2^l = \mu^l, l=1,2.
\right\}.
\end{equation*}

Let $\phi$ be a smooth, canonical bump function supported in $B_1(0)$, constantly equal to 1 in $B_{1/2}(0)$, and satisfying $\phi(x) \lesssim \operatorname{dist}(x, \partial B_1)^2$. We define $\phi_r := \phi(\cdot/r)$.

We define the stronger excess as
\begin{equation*}\label{e:strongexcess}
  \mathbb{E}(T,C,\bB_r) := r^{-(m+2)} d(\phi_r d|\!|T|\!|, \phi_r d|\!|C|\!|).
\end{equation*}

We are going to restrict ourselves to a narrower setup which will be more than enough to produce an example of an essential boundary singularity. 
\begin{assumption} \label{a:corneredboundary}
Suppose we are under the Assumptions of 
\ref{a:corneredboundarymonotonicity}.
Let $N_{\Gamma_i^j}(p)$ denote the vector in $T_p(\Gamma_i^j)$ which is normal $M$.
\begin{itemize}
\item We assume that for every $p \in M \cap \bB_1$, $\forall 1 \leq i_0 \leq N^0$ and $\forall 1 \leq i_1 \leq N^1$
\begin{equation*}
N_{\Gamma_{i_0}^0}(p) \perp N_{\Gamma_{i_1}^1}(p).
\end{equation*}
\item We also assume that there is a universal constant $\alpha_0$ such that for every $(i_0,j_0) \neq (i_1,j_1)$ and $1 \leq i_0 \leq N^{j_0}$, $1 \leq i_1 \leq  N^{j_1}$, $0 \leq j_0 \leq j_1$
\begin{equation*}
2^{1000}\alpha_0\leq \left \langle N_{\Gamma_{i_0}^{j_0}}(p), N_{\Gamma_{i_1}^{j_1}}(p)  \right \rangle.
\end{equation*}
\end{itemize}
\end{assumption}

The angle condition is a reasonable assumption since $\Gamma_i^j$ are assumed to be smooth and transversal. In practice this only will mean that we will not track the fine dependence of the angle in the Lipschitz approximations and decay, which is irrelevant for the applications we prove in this paper.
This setting is different than \cite{ian2024uniqueness}, since the tangent cone is automatically unique by soft arguments since the boundary data does not give any degree of freedom. The interesting question in this setting instead is the rate of decay which we do need for the application in mind. 

Assumption \ref{a:corneredboundary} provides a good setup for this problem because of the following lemma.

\begin{lemma}\label{l:uniquetangent}
Let $T$, $\Gamma_i^j$, $M$, $\Sigma$,  be as in Assumption \ref{a:corneredboundary}. Then $\Theta(T,p) \geq Q/4$, in the equality case there is a unique tangent cone at $T$ which is a cornered open book.
\end{lemma}
The importance of obtaining the lower density bound is that we will use Leon Simon's uniqueness of cylindrical tangent cones strategy \cite{simon1993cylindrical} and this provides us a no good density gaps condition. 
\begin{remark}A slightly more general version of Lemma \ref{l:uniquetangent} holds with no modification. Under Assumption \ref{a:corneredboundarymonotonicity}, if every tangent cone at $p$ is of the form
\begin{equation*}
    C=\sum_{i=1}^Q \a{H_i}
\end{equation*} where each $H_i$ is an open subset of a plane, then the tangent cone is unique.
\end{remark}

\begin{proof}[Proof of Lemma \ref{l:uniquetangent}]
We refer to Definition \ref{d:corneredopenbookboundarydata} of cornered open books with the right boundary data. 

Given $C \in \mathcal{C}_p$, it is easy to see that $\mathbb{E}(T,C,\bB_r)$ is continuous on $r$. We know that $\mathcal{C}_p$ is a finite set, since there are finitely many possibilities for choosing $Q$ quadrants with the right boundary data. Each of them corresponds to a function $l:\left \{1,...,Q \right\} \rightarrow \left\{1,...,N^0\right\} \times \left\{1,...,N^1\right\}$ which associates a quadrant $H_i$ with it's boundary $T_p(\Gamma_{l(i)_0}^0)$ with positive orientation and $T_p(\Gamma_{l(i)_1}^1)$ with negative orientation (where $l(i)_0$,$l(i)_1$ are the coordinates of $l(i)$). The function uniquely determines the cornered open book. This implies that $\mathcal{C}_p$ is a finite set and thus it must be disconnected in the Wasserstein distance.

We consider the function $f: (0,1) \rightarrow \RR^{\# \mathcal{C}_p}$ given by

\begin{equation*}
f(r):=(\mathbb{E}(T,C,\bB_r))_{C \in \mathcal{C}_p}.
\end{equation*}

We consider all the possible limits of the function $f(r)$ along a sequence of radius going to zero. This must be a connected subset of $\RR^{\# \mathcal{C}_p}$ because $f$ is a continuous function of $r$ and the limit set of a continuous function is connected.

We know that there are $\#\mathcal{C}_p$ possible values the limit of $f$ can take which is 

\begin{equation*}
(\mathbb{E}(C,C',\bB_1))_{C \in \mathcal{C}_p}.
\end{equation*}

This is a disconnected set of $\RR^{\# \mathcal{C}_p}$ and thus the only option is for it to have a single element and thus a single tangent cone.
\end{proof}
\section{Linear problem}
In this section, we introduce an appropriate linear model for the nonlinear problem with singular boundaries of the type considered in this paper. We refer to \cite{DS1} for the detailed setup for the linear problem in the interior and Section 8 of \cite{ian2024uniqueness} for the boundary setting.

Given $\Omega \subset \RR_+^m$ a Lipschitz domain there is a well defined continuous trace operator 
\begin{equation*}
    \circ|_{\partial \Omega}:W^{1,2}(\Omega,\mathcal{A}_Q(\RR^n)) \rightarrow  L^2(\partial \Omega, \mathcal{A}_Q(\RR^n)).
\end{equation*}

\begin{definition}
Let $\Omega$ be a Lipschitz domain. We say $u$ is Dir-minimizing in $\Omega$ if
\begin{equation*}
    u \in W^{1,2}(\overline{\Omega},\mathcal{A}_Q(\RR^n))
\end{equation*}
and for every $v \in  W^{1,2}(\overline{\Omega},\mathcal{A}_Q(\RR^n)): u \res \partial \Omega = v\res \partial \Omega$  \begin{equation*}
\int_ {\Omega} |Du|^2 \leq \int_{\Omega}|Dv|^2.
\end{equation*}
\end{definition}

We define the quarter balls 
\begin{equation*}
\bB_{r}^{1/4}:=\bB_r \cap  \left(\RR^+ \times \RR^+\times \RR^{m-2} \right)
\end{equation*}

In order to keep the notation consistent with the rest of this work, we define the quarter plane $H:= \RR^+ \times \RR^+ \times \RR^{m-2}$. We denote the boundary as $V^0=\left\{0\right\} \times \RR^+ \times \RR^{m-2}$ and $V^1:=\RR^+  \times \left\{0\right\} \times \RR^{m-2}$. We denote the corner as $L:=\left\{0\right\} \times \left\{0\right\} \times \RR^{m-2}.$

In this work it will be enough for us to consider the linear problem on the quarter space. More precisely we consider the following setup. 
\begin{assumption}[Linear problem on the quarter space]\label{a:linearproblemquarterspace}
Given a Dir-minimizing function $u \in W^{1,2}(\bB_1^{1/4},\mathcal{A}_Q(\RR^n))$. The boundary linear problem on the quarter space consists of functions $u$ such that $u \res \left(V^0 \cup V^1 \right)\cap \bB_1=Q\a{0}$.
\end{assumption}
We adapt the classical Almgren frequency function to this quarter-domain setting at points in the corner $L$.
\begin{definition}
For $x \in L \cap \bB_1$, $0<r<1-|x|$, we define the Dirichlet energy, $L^2$ spherical height and the frequency function:
\begin{equation*}
    D(x,r)=\int_{\bB_r^{1/4}(x)} |Du|^2, \; H(x,r)= \int_{\partial \bB_r^{1/4}(x)} |u|^2 \; \textup{and} \; I(x,r)= \frac{rD(x,r)}{H(x,r)}
\end{equation*}
\end{definition}

The arguments here follow those of Section 8 in \cite{ian2024uniqueness}, which studies the higher multiplicity boundary linear problem.
\begin{theorem} \label{t:monotonfrequency} The frequency function for a function $u$ satisfying Assumption \ref{a:linearproblemquarterspace} is monotone. For any $x \in L \cap \bB_1$ either there exists $\rho$ such that $u|_{B_{\rho}(x)}\equiv 0$ or $I(x,r)$ is an absolutely continous non decreasing positive function of $r$ on $(0,1-|x|)$.
\end{theorem}
Observe a minor change here with \cite{ian2024uniqueness} and \cite{DS1}. In this setting, we don't claim uniform convergence, since have not proved Hölder continuity of $u$ before this. We are not going to address the Hölder continuity for functions $u$ satisfying Assumption \ref{a:linearproblemquarterspace} since we won't need it.
\begin{theorem} 
 Let $f \in W^{1,2}(\bB^{1/4}_1,\mathcal{A}_Q)$ be Dir-minimizing with zero boundary data and $\textup{Dir}(f,B_{\rho})>0$ for every $\rho \leq 1$. Then for any sequence $\rho_k \rightarrow 0$ there exists a subsequence, not relabeled, such that $f_{0,\rho}$ coverges locally strongly in $W^{1,2}(H)$ to a Dir-minimizing function $g:H \rightarrow \mathcal{A}_Q(\RR^n)$ with the following properties:
 \begin{enumerate}
     \item $\textup{Dir}(g,\bB_1^{1/4})=1$ and $g| \Omega$ is Dir-Minimizing for any bounded $\Omega$.
     \item $g(x)=|x|^{\alpha}g\left(\frac{x}{|x|}\right)$, where $\alpha=I_{0,f}(0)$ is the frequency of $f$ at $0$.
 \end{enumerate}
\end{theorem}

The following two lemmas have analogous proofs to those in Section 8 of \cite{ian2024uniqueness} for the linear problem. We will use them to lower bound the frequency in Lemma \ref{l:frequencybound}.
\begin{lemma}{(Cylindrical Blowups)}\label{l:cylindrical}
Assume that $m \geq 3$. Let $g \in W^{1,2}(H,\mathcal{A}_Q(\RR^n))$ be an $\alpha$ homogeneous and Dir-Minimizing function with zero boundary value at $V^0 \cup V^1$, $\textup{Dir}(g,\bB_1^{1/4})>0$ and $\beta=I_{z,g}(0)$. Suppose $z=\frac{e_m}{2}$ ($z \in L$). Then, the tangent functions $h$ to $g$ at $z$ satisfy
\begin{equation*}
h(x_1,...,x_{m-1},x_m)=\hat{h}(x_1,...,x_{m-1})
\end{equation*} where $\hat{h} \in W^{1,2}(H \cap \RR^{m-1} \times \left\{0 \right\},\mathcal{A}_Q(\RR^n))$ is a Dir-Minimizing Q valued function with zero boundary values on $\RR^{m-1}_+$
\end{lemma}
\begin{lemma}\label{Lem:Invariance}Let $h$ be a $\alpha$ homogeneous Dir-minimizing function as in Assumption \ref{a:linearproblemquarterspace}. Suppose that $I(z)=\alpha$ for $z=e_m/2$. Then $h(x,s)=\hat{h}(x)$ for a Dir-minimizing function $\hat{h}$ on $H \cap \RR^{m-1} \times \left \{ 0\right\}$ with zero boundary data at $\left( V^0 \cup V^1 \right)\cap \left(\RR^{m-1} \times \left\{0\right\} \right)$.
\end{lemma}

\begin{lemma}\label{l:frequencybound} Let $u$ be as in Assumptions of the linear problem on the quarter space \ref{a:linearproblemquarterspace}. Assume $u$ is $\alpha$-homogeneous. Then $\alpha \geq 2$. Furthermore if $u$ is $2$-homogeneous then if $\theta$ denotes the angle on the quarter-plane (in the first two components)
\begin{equation*}
u(x)= \sum_{i=1}^k \a{v_i \cos(2\theta)|x|^2}.
\end{equation*}
\end{lemma}
\begin{proof} Assume by contradiction there exists $u$ an $\alpha$ homogeneous Dir-minimizing function with $\alpha<2$. 

Let $z \in L$. By homogeneity, if $z \neq 0$ then for all $r>0$, $I(rz)=I(z)$. Thus by upper-semi-continuity  $\alpha \geq I(z)$. We take a blowup at $z$ which gives us an $I(z) \leq \alpha<2$ homogeneous Dir-minimizer on $H \cap \RR^{m-1}$. By iterating and applying Lemma \ref{l:cylindrical}, we obtain a $\beta$ homogeneous Dir-minimizing function $u$ on $\RR^+ \times \RR^+$ with zero boundary data on $\left \{0 \right\} \times \RR \cup \RR \times \left \{0 \right\}$.

By the interior regularity for Dir-minimizing functions $u$ \cite{DS1} in $2$d the function $u$ needs to be smooth outside of its boundary. Thus it must decouple as classical harmonic function that has zero boundary value along $\left \{0 \right\} \times \RR \cup \RR \times \left \{0 \right\}$. A classical harmonic function which vanishes on the boundary of the quarter $2$d plane must have frequency at least $2$. This is because it's restriction to $S^1$ must solve 
\begin{equation*}
f''(\theta)+\lambda^2 f(\theta)=0
\end{equation*}
with $f(0)=f(\pi/2)=0$. We can extend periodically, which makes it a classical homogeneous harmonic function in the plane. Since $f$ constant or $f$ linear are not options for $f$ to be non-trivial, the homogeneity for $f$ must be at least $2$.
Moreover we must have that $\lambda=k$ is an even integer and
\begin{equation*}
f(\theta)=\sum_{i=1}^Q \a{v_i \cos(k\theta)}. 
\end{equation*}
If $k \neq 2$ then $f$ has interior singularities where $\cos$ vanishes and thus $f$ must be $2$-homogeneous and in this case for $v_i \in \RR^{n}$
\begin{equation*}
f(\theta)=\sum_{i=1}^Q \a{v_i \cos(2\theta)}. 
\end{equation*}

When $u$ is $2$-homogeneous, an argument analogous to the one at the end of the proof of Lemma \ref{l:densitybound} shows that $u$ arises from a two-dimensional solution in a quarter space.
\end{proof}

We have an analogue of Theorem 8.19 \cite{ian2024uniqueness} in this setting. We omit the proof since it is analogous.
\begin{theorem}[Local Dir-minimizers] \label{t:localdirminimizers}
Let $u \in W^{1,2}(\bB_1^{1/4},\mathcal{A}_Q(\RR^n))$ have zero boundary value along $V_0 \cup V_1$.
Assume that $u$ is Dir-minimizing on every subdomain of $\bB_1^{1/4}$ with smooth boundary. Then $u$ is Dir-minimizing in $\bB_1^{1/4}$ with zero boundary value along $V_0 \cup V_1$.
\end{theorem}

\begin{proposition}\label{p:heightdecay}
If $I(0)=\alpha$ then
\begin{equation*}
\int_{\bB_r^{1/4}}|u|^2 \leq\frac{1}{m+1}r^{m-1+2\alpha} \int_{\partial\bB_1^{1/4}}|u|^2.
\end{equation*}
In particular since $I(0) \geq 2$ then
\begin{equation*}
\limsup_{r \rightarrow 0}\frac{1}{r^{m+4}} \int_{\bB_r^{1/4}} |u|^2<\infty.
\end{equation*}
\end{proposition}
\begin{proof}
    The proof is analogous to Proposition 8.11 of \cite{ian2024uniqueness}. It follows from the estimates on the spherical height which are a consequence of the monotonicity of the frequency. See Chapter 3.4 of \cite{DS1}.
\end{proof}
\section{Lipschitz approximation}
In this section, we construct the Lipschitz approximation and conclude with some technical estimates regarding measure theoretic excess and its relationship to $L^2$ excess.

The general assumption for the Lipschitz approximation is the following:
\begin{assumption}\label{A:Lipschitzapproximation} Suppose that $T$, $\Gamma_i^j$, $\Sigma$, $M$
 are as in Assumption \ref{a:corneredboundary}, but now defined in the ball $\bB_8$. Let $C$ be the unique tangent cone at $0$. For a small constant $\varepsilon(Q,m,n,\overline{n})$ we assume
 \begin{equation*}
 \mathbb{E}(T,C,\bB_8)+ \mathbf{A_{\Gamma}}+\mathbf{A}_{\Sigma}^2 \leq \varepsilon \alpha(C)^2.
\end{equation*}
Moreover we assume that for every $p \in M \cap \bB_2$ and all $0<r<1/8$
\begin{equation*}
|\!|T|\!|(\bB_r(p)) \leq \left(Q/4 +1/16\right) \omega_m r^m.\end{equation*}
 \end{assumption}

We consider a parameter $\delta>0$ which will be taken small enough depending on $Q_i^j$, $m$, $n$, $\alpha_0$.
In order to construct the Lipschitz approximation we will need to define a suitable Whitney region. Let $L=\RR^{m-2} \times \left\{0\right\}$. Let $P_0$ be a cube with side length $\frac{2}{\sqrt{m-2}}$ centered at $0$.

We define the regions as in \cite{de2023fineIII}. The set $R:=\left\{p: \mathbf{p}_{L}(p) \in P_0 \: \textup{and} \: 0<|\mathbf{p}_{V^{\perp}}(p)| \leq 1 \right\}$. The cubes $\mathcal{G}_l$ are cubes of sidelength $\frac{2^{1-l}}{\sqrt{m-2}}$ obtained by subdividing $P_0$ into $2^{l(m-2)}$ cubes. The integer $l$, the generation of the cube, is denoted by $l(P)$. If $P \subset P^{\prime}$ and $\ell\left(P^{\prime}\right)=\ell(P)+1$, we then call $P^{\prime}$ the parent of $P$, and $P$ a child of $P^{\prime}$. When $\ell\left(P^{\prime}\right)>\ell(P)$, we say that $P^{\prime}$ is an ancestor of $P$ and $P$ is a descendant of $P^{\prime}$.

For every cube $P$ we denote by $y_P$ its center and $\mathbf{B}(P)$ the ball $\mathbf{B}_{2^{2-l(P)}}(y_P)$ (in $\RR^{m+n}$) and by $\mathbf{B}^h(P)$ the set $\mathbf{B}(P) \setminus B_{2^{-5-l(P)}}(L)$.

We define for $P \in \mathcal{G}_l$
\begin{equation*}
R(P):=\left\{p: \mathbf{p}_{L}(p) \in P \: \textup{and} \: 2^{-l-1} \leq |p_{L^{\perp}}(p)| \leq 2^{-l} \right\}.
\end{equation*}

We also define $\lambda P$ as the cube concentric with $P$ with side length $\lambda \frac{2^{1-l}}{\sqrt{m-2}}$ and
\begin{equation*}
\lambda R(P):=\left\{p: \mathbf{p}_{L}(p) \in \lambda P \: \textup{and} \: \lambda^{-1}2^{-l-1} \leq |p_{L^{\perp}}(p)| \leq \lambda 2^{-l} \right\}.
\end{equation*}
We also define $\lambda P_i:=\lambda R(P) \cap H_i$.
We notice that $\mathbf{p}_{\hat{H}_i}(\Gamma_{l(i)}^0)$ and $\mathbf{p}_{\hat{H}_i}(\Gamma_{l(i)}^1)$ delimit a domain in $\hat{H}_i$. When $\mathbf{A_{\Gamma}}$ is small, we must have that this domain must be close to $H_i$, a quarter of the plane $\hat{H}_i$.  We denote $D_i$ the described domain.

\begin{lemma}[Splitting]\label{l:splitting}
There exists a constant $\delta(Q,m,n,\alpha_0)$  with the following property. Assume that
\begin{equation*}
\int_{\bB_{4} \setminus B_{1/32}(L)}\dist(x,C)^2 d|\!|T|\!| + \mathbf{A_{\Gamma}} + \mathbf{A}_{\Sigma}^2 \leq \delta^2.
\end{equation*}
 There exist $T_1,T_2,...,T_N$ such that \begin{equation*}
    T \res \left(\bB_{3.5} \setminus B_{3/64}(L)\right) = \sum_{i=1}^N T_i \; \textup{and } \; \forall i \; :\partial T_i= Q_i'\a{\Gamma_{l(i)}^0}-Q_i'\a{\Gamma_{l(i)}^1}.
\end{equation*}
The multiplicity of $Q_i'$ is the one obtained by constancy over the projection on  $\left(\mathbf{p}_{H_i}\right)_{\#} T_i$.
Furthermore, if instead we assume 
\begin{equation*}
\mathbb{E}(T,C,\bB_8) + \mathbf{A_{\Gamma}}+\mathbf{A}_{\Sigma}^2<\delta^2
\end{equation*}
then $Q_i'=Q_i.$
\end{lemma}
\begin{proof}

We will show that if $\delta$ is small enough there exist disjoint open sets $U_i \subset \bB_{3.5} \setminus B_{3/64}(L)$ that satisfy
\begin{equation*}
\partial{U_i} \cap \partial{U_j} \subset  \mathbf{\Gamma} \cap  \left(\bB_3 \setminus B_{1/32}(L)\right)
\end{equation*}
and 
\begin{equation*}
\textup{spt}(T) \cap \bB_{3.5} \setminus B_{3/64}(L)\subset \cup_{1 \leq i \leq N} U_i.
\end{equation*}

We start by taking a cover of $\mathbf{\Gamma}\cap \bB_{3.5} \setminus B_{3/64}(L)$ by balls $\bB_{2^{-12}\alpha_0}(p_k)$, which we assume to be a Besicovitch covering. By the choice of the radii, if $\mathbf{A_{\Gamma}}$ is small enough no ball will intersect two different boundary pieces. 

We consider
\begin{equation*}
 A:= \left(\textup{spt}(T) \cap \left(\bB_{3.5} \setminus B_{3/64}(V) \right)\right) \setminus B_{2^{-20}\alpha_0}(\Gamma).
\end{equation*}

We must have assuming that $\mathbf{A_{\Gamma}}$ is small enough, by the height bound that
\begin{equation*}
\sup_{x \in A} \dist(x,C) \leq C \delta^2.
\end{equation*}

If $\delta$ is small enough with respect to $\alpha_0$ and $m,n$ then $A$ is split into disjoint neighborhoods $U_{i,0}$ of $H_i$. The neighborhoods $U_{i,0}$ serve to decompose $\textup{spt}(T)$ away from $\mathbf{\Gamma}$.

By the bounded overlap property, we know that 
\begin{equation}
\sum_l \int_{\bB_{2^{-12}\alpha_0}(p_k)} \dist(x,C)^2 d|\!|T|\!| \leq C(\alpha_0) \int_{\bB_4 \setminus B_{1/32}(L)} \dist(x,C)^2 d|\!|T|\!| \leq \delta^2.
\end{equation}

At the boundary $\mathbf{\Gamma}$ in $\bB_3 \setminus B_{3/64}(V)$, we use the following Theorem which we state combining two results:
 \begin{theorem}[Decomposition: Theorem 5.13 \cite{ian2024uniqueness} + Theorem 5.15 \cite{ian2024uniqueness}] \label{T:decomposition-1sided}
Let $T$ and $\Gamma$ be under \ref{A:generalsingleboundary}. Assume that $C$ is a non-flat open book whose spine is $T_0(\Gamma)$ There exists $\varepsilon>0$ such that if 
\begin{equation*}
\max\left\{\mathbf{E}(T,C,\bB_1(0)), \varepsilon^{-1}\mathbf{A}_{\Gamma},\varepsilon^{-1}\mathbf{A}_{\Sigma}^2\right\}\leq \varepsilon \alpha(C)^2.
\end{equation*}
There there exist $T_1,T_2,..,T_N$ area-minimizing currents and $Q_i$ integers on $\bB_{1/64} \cap \Sigma$ with
\begin{itemize}
\item $\partial T_i=Q_i \a{\Gamma}$ 
\item \begin{equation}
T\res\bB_{1/64}(0)=\sum_{i=1}^N T_i
\end{equation}
\item The supports of $T_i$ only intersect at $\Gamma$.
\end{itemize}
\end{theorem}

In our setting, we will apply it to a modification of $C$ the starting cornered open book. All cornered open books with the right boundary data $C'$ (recall Definition \ref{d:corneredopenbookboundarydata})  under Assumption \ref{a:corneredboundary} satisfy $\alpha(C') \gtrsim \alpha_0$.
We need to translate and adjust the cone $C$ so that it passes through the center where we apply this result, but that only introduces errors that are controlled by $\mathbf{A_{\Gamma}}$ and keeps $\alpha(C) \gtrsim \alpha_0.$

If $\delta$ is small enough then we will be able to apply the theorem above and get such decomposition at scale $\bB_{2^{-18}}(p_k)$.  This means that $T \res \bB_{2^{-18}}(p_k)=\sum_{1 \leq i \leq N}T_i$ and we can choose $U_{i,l}$ disjoint neighborhoods of $\bB_{2^{-18}}(p_k)$ each of which containing $\left(H_i \cup \textup{spt}(T_i)\right) \cap \bB_{2^{-18}}(p_k).$

We define 
\begin{equation*}
    U_i=U_{i,0} \bigcup_{l} U_{i,l}
\end{equation*}
which satisfy the desired properties.

The last two remarks follow are a consequence of the construction of the strong Lipschitz approximation in the interior and of the definition of the excess. See Proposition 6.3 \cite{ian2024uniqueness} and Lemma 8.5 \cite{de2023fineIII}.
\end{proof}

\begin{lemma}[Lipschitz Approximation]\label{l:lipschitzapprox}
Let $T$ and $C$ as in the above Lemma \ref{l:splitting}. Let \begin{equation*}
E_i:= \int_{\bB_3 \setminus B_{1/32}(L)}\dist(p,H_i)^2 d|\!|T_i|\!|(p).
\end{equation*}
Let $\Omega_i:=D_i \cap \bB_2 \setminus B_{1/16}(L)$ and $\mathbf{\Omega}_i:=\bB_3 \cap \mathbf{p}_{\hat{H_i}}^{-1}(\Omega_i).$ 
\begin{enumerate}
    \item We have $\left(p_{\hat{H_i}}\right)_{\#}(T_i)=Q_i' \a{D_i}$.
    \item 
Then there exist Lipschitz multivalued functions $u_i: \Omega_i \rightarrow A_{Q'_i}(\hat{H}_i^{\perp})$ and a closed set $K_i \subset \Omega_i$ such that $\textup{gr}(u_i) \subseteq \Sigma$, $T_i \res \mathbf{p}_{H_i}^{-1}(K_i)=\mathbf{G}_{u_i}\res \mathbf{p}_{H_i}^{-1}(K_i)$ and the following estimates hold
\begin{equation*}
u_i \res \left(\mathbf{p}_{\hat{H_i}}(\Gamma_{l(i)}^j) \cap \bB_2 \setminus B_{1/16}(L)\right)=Q_i' \a{\mathbf{p}_{\widehat{H}_i^{\perp}}(\Gamma_{l(i)}^j)} \: j=0,1.
\end{equation*}
\begin{equation*}
\left\|u_i\right\|_{\infty}^2+\left\|Du_i\right\|_{L^2}^2 \leq CE_i +C\mathbf{A_{\Gamma}}+C\mathbf{A}_{\Sigma}^2
\end{equation*}
\begin{equation*}
\textup{Lip}(u_i) \leq C\left(E_i+\mathbf{A_{\Gamma}}+\mathbf{A}_{\Sigma}^2\right)^{\gamma}
\end{equation*}
\begin{equation*}
|\Omega_i \setminus K_i| + |\!|T|\!|(\mathbf{\Omega_i} \setminus p_{H_i}^{-1}(K_i))\leq C\left(E_i+\mathbf{A_{\Gamma}}+\mathbf{A}_{\Sigma}^2\right)^{1+\gamma}.
\end{equation*}
\item If additionally $\mathbb{E}(T,C,\bB_8)<\delta^2$ then $Q_i=Q_i'$ for all $i$.
\end{enumerate}
\end{lemma}
\begin{proof}
The proof is standard so we will just give some basic considerations. The new element in this context is that we do take the boundary data. The proof follows as in Proposition 8.6 \cite{de2023fineIII}, but we need to incorporate the boundary and for that we use Proposition 5.1 \cite{ianreinaldo2024regularity}.The Lipschitz approximation gets defined in balls which overlap and then patched together. Given a point $p \in \mathbf{\Gamma} \cap \bB_{3.5} \setminus B_{3/64}(L)$ in order to use Proposition 5.1 \cite{ianreinaldo2024regularity} we need to slightly tilt the plane to a new plane $\hat{H}_{i,p}$ so that for the appropriate $l,j$,  $T_p(\Gamma_l^j)\subset \hat{H}_{i,p} \subset T_p(\Sigma).$
\end{proof}
The analogous statement holds locally over Whitney regions associated to a cube $P$. We define
\begin{equation*}
\mathbf{E}(P):=2^{(m+2)\ell(P)}\int_{\bB^{h}(P)} \dist(x,C)^2d|\!|T|\!|.
\end{equation*}

\begin{lemma}[Local Splitting]\label{l:localseparation}
There exists a constant $\delta$ such that if 
\begin{equation*}
\mathbf{E}(P)+ \mathbf{A_{\Gamma}} + \mathbf{A}_{\Sigma}^2 \leq \delta^2.
\end{equation*}
Then there exist  $T_{1,P},...,T_{N,P}$ such that \begin{equation*}
    T \res \left(\bB_{3 \cdot 2^{-\ell(P)}}(y_p) \setminus B_{2^{-4-\ell(P)}}(L)\right) = \sum_{i=1}^N T_i \; \textup{and } \; \forall i \; :\partial T_i= Q_i'\a{\Gamma_1^0}-Q_i'\a{\Gamma_i^1}.
\end{equation*}
The multiplicity of $Q_i'$ is the one obtained by constancy over the projection on  $\left(\mathbf{p}_{H_i}\right)_{\#} T_i$.
Furthermore, if we assume 
\begin{equation*}
\mathbb{E}(T,C,\bB_8) + \mathbf{A_{\Gamma}}+\mathbf{A}_{\Sigma}^2<\delta^2
\end{equation*}
then $Q_i'=Q_i$ for every $P$ with $\ell(P) \sim 1$.
\end{lemma}
\begin{definition}\label{d:deltazero}
Let $\delta_0=\delta(Q,m,n,\overline{n})$ be so that Lemma \ref{l:lipschitzapprox} and \ref{l:localseparation} hold.
\end{definition}
\begin{definition} Let $P \in \mathcal{G}$. We say that 
\begin{enumerate}
\item A cube $P$ is a \textit{classical boundary cube} if $\mathbf{B}^h(P') \cap M=\varnothing$
for every ancestor $L'$ of $L$  (including $L$) and 
\begin{equation}\label{eq:densitycondition}
|\!|T|\!|(\mathbf{B}(L)) \leq (Q/4+1/4)2^{m(2-l(P))}.
\end{equation}
\item A cube $P$ is a \textit{corner stopping cube} if $\mathbf{B}^h(P) \cap M \neq \varnothing$ or $|\!|T|\!|(\mathbf{B}(P)) >(Q/4+1/4)2^{m(2-l(P))}$ and it's parent is a classical boundary cube.
\item A cube $P$ is an \textit{outer cube} if it is of classical boundary type and $\mathbf{E}(P')\leq \delta_0^2$ (with $\delta_0$ as in \ref{d:deltazero}) for every ancestor $P'$ of $P$ (including $P$).  We denote $\mathcal{G}^o$ the outer cubes.
\item A cube $P$ is a \textit{inner cube} if it is a classical boundary cube, but it an outer cube.
\end{enumerate}
It is convenient to define the following regions
\begin{itemize}
    \item The \textit{corner region region} denoted $R^{corner}$ which is the union of $R(P)$ over all cubes which are not classical boundary cubes.
    \item The \textit{outer region} denoted $R^o$, is the union of $R(P)$ over all outer cubes. 
    \item The \textit{inner region} denoted $R^i$, is the union of $R(P)$ over all classical boundary cubes which are not outer cubes.
    \end{itemize}
\end{definition}
\begin{definition}
Let $T, \Gamma_i^j, \Sigma, M, C$ be as in Assumption \ref{A:Lipschitzapproximation}. Fix $P \in \mathcal{G}^o$. We denote
\begin{equation*}\overline{\mathbf{E}}(P):=\max \left\{\mathbf{E}\left(P^{\prime}\right): P^{\prime} \in \mathcal{G}^o \: \textup{with} \: R(P) \cap R\left(P^{\prime}\right) \neq \emptyset\right\}
\end{equation*}
\end{definition}
We use Lipschitz approximation locally and then patch the pieces as in \cite{de2023fineIII} to get a coherent outer approximation.
\begin{proposition}[Coherent outer approximation] Let $T, \Gamma, \Sigma, C$ as in Assumption \ref{A:Lipschitzapproximation}. We define
\begin{equation*}
R_i^o:= \cup_{L \in \mathcal{G}^o}L_i \equiv H_i \cap \cup_{L \in \mathcal{G}^o} R(L).
\end{equation*}
Then, there are Lipschitz multivalued maps $u_i: R_i^o \rightarrow A_{Q_i}(\widehat{H}_i^{\perp})$ and closed subsets $\overline{K_i}(P) \subset P_i$ satisfying the following properties:
\begin{enumerate}
\item \begin{equation*}
u_i \res \left(\mathbf{p}_{\hat{H_i}}(\Gamma_{l(i)}^j) \cap R_i^o\right)=Q_i \a{\mathbf{p}_{\widehat{H}_i^{\perp}}(\Gamma_{l(i)}^j)} \: j=0,1.
\end{equation*}
\item 
$gr(u_i) \subset \Sigma$ and $T_{P,i} \res \mathbf{p}_{H_i}^{-1}(\overline{K}_i(P))=\mathbf{G}_{u_i} \res \mathbf{p}_{H_i}^{-1}(\overline{K}_i(P))$ for every $L \in \mathcal{G}^o$, where $T_{P,i}$ is defined as in the proposition above.
\item \begin{equation*}
2^{2l(P)}\left\|u_i\right\|_{L^\infty(L_i)}^2+2^{ml(P)}\left\|Du_i\right\|_{L^2(L_i)}^2 \lesssim \overline{\mathbf{E}}(L)+ 2^{l(P)}\mathbf{A_{\Gamma}}+2^{-2l(P)}\mathbf{A}_{\Sigma}^2
\end{equation*}
\begin{equation*}
\left\|Du_i \right\|_{L^{\infty}(L_i)} \lesssim \left(\overline{\mathbf{E}}(L)+\mathbf{A}_{\Sigma}^2\right)^{\gamma}
\end{equation*}
\begin{equation*}
|L_i \setminus K_i(L)| + |\!|T_{L,i}|\!|(\mathbf{p}_{H_i}^{-1}(L_i \setminus K_i(L))\lesssim 2^{-ml(P)}\left(\overline{\mathbf{E}}(L)+2^{-2l(P)}\mathbf{A}_{\Sigma}^2\right)^{1+\gamma}
\end{equation*}
\end{enumerate}
\end{proposition}

\begin{proof}
The proof is the same as Proposition 8.19 in \cite{de2023fineIII}. The coincidence set is taken to be $\overline{K}_{i}(P):=\cap_{P': R(P') \cap R(P) \neq \varnothing}K_i(P')$. The Lipschitz function  $u_i$ is obtained by extending the function from $\cup_{P \in R^o}\overline{K}_i$.
\end{proof}

In the following lemma we consider the definitions of the different type of cubes as above, but we do not require the global Assumption \ref{A:Lipschitzapproximation}. We still consider the definition of types of cubes as above with respect to the cone $C$ and their associated Lipschitz approximation. 
\begin{lemma}\label{l:reverseexcessbound}Let $T, \Gamma, \Sigma$ be as in Asumption \ref{a:corneredboundary} in $\bB_4$. Assume that every cube that intersects $\bB_{2^{-10}}$ and their ancestors are outer cubes whose multiplicities of their local Lipschitz approximations agree with those of $C$. 
\begin{equation*}
\mathbb{E}(T,C,\bB_1) \leq  C(Q,m,n,\overline{n},\delta_0) \left(\mathbf{E}(T,C,\bB_4) + \mathbf{A}_{\Sigma}^2+\mathbf{A_{\Gamma}}^2\right) 
\end{equation*}
\end{lemma}

\section{Proof of main theorems at corner points}
In this section, we will show a decomposition theorem, excess decay and Hölder Continuity of the normal map.
We will define what we mean by a regular point at $M$.
\begin{definition}
Let $T, \Gamma_i^j, \Sigma$, $Q_i^j$ be as in Assumption \ref{a:corneredboundary}. A point $p \in M$ is \textit{regular} if
\begin{equation*}
T \res \bB_{\rho}(p)=\sum_{i=1}^N Q_i \a{\Xi_i}
\end{equation*}
where $\sum_{i=1}^N Q_i=Q$ and 
\begin{itemize}
\item $\Xi_i$ is an area-minimizing surface in $\bB_{\rho}(p)$ with
\begin{equation*}
\partial \Xi_i= \a{\Gamma_{l(i)}^0}-\a{\Gamma_{l(i)}^1}.
\end{equation*}
Moreover, $\Xi_i$ in $\bB_{\rho}(p)$ is smooth outside of $M$ and attaches smoothly to both $\Gamma_{l(i)}^0$ and $\Gamma_{l(i)}^1$. 

\item For $i_1 \neq i_2$ we have that
\begin{equation*}
    \textup{spt}(\Xi_{i_1}) \cap \textup{spt}(\Xi_{i_2}) \subset \left(\bigcup_{j \in \left\{0,1 \right\}}\Gamma_{l(i_1)}^j \cap \Gamma_{l(i_2)}^j  \cap \bB_{\rho}(p)\right)
\end{equation*}
where if $\Gamma_{l(i_1)}^j= \Gamma_{l(i_2)}^j$ then
$\Gamma_{l(i_1)}^j \cap \Gamma_{l(i_2)}^j= \Gamma_{l(i_1)}^j=\Gamma_{l(i_2)}^j$, otherwise $\Gamma_{l(i_1)}^j= \Gamma_{l(i_2)}^j=M$.
\end{itemize}
\end{definition}
\begin{theorem}\label{t:decomposition}
Let $T, \Gamma_i^j, \Sigma$, $Q_i^j$ be as in Assumption \ref{a:corneredboundary}. Assume that $\Theta(T,0)=Q/4$. By Lemma \ref{l:uniquetangent} there is a unique tangent cone, which is a cornered open book $C \in \mathcal{C}_0$. Then there is a neighborhood $\bB_{\rho}$ of $0$ such that $T \res \bB_{\rho}=\sum_{1 \leq i \leq N} T_i$, each $T_i$ with corresponding flat tangent cone $Q_i\a{H_i}$ (each $H_i$ a single quadrant from the cornered open book), with  \begin{equation*}
       \partial T_i=Q'_i\a{\Gamma_{k(i,0)}^0}
       -Q'_i\a{\Gamma_{k(i,1)}^1}.
    \end{equation*}  
In the case $N^0=1$, $N^1=N$, $Q_1^0=N$, $Q_1^1=Q_2^1=...=1$, $0$ must be a regular point for $T$.
\end{theorem}

\begin{lemma}[Excess decay with small second fundamental forms]\label{l:excessdecay}
Let $T, \Gamma_i^j, \Sigma$, $Q_i^j$ be as in Assumption \ref{a:corneredboundary}.
There $\varepsilon$ and $\rho<1/2$ such that if $\mathbb{E}(T,C,\bB_1)<\varepsilon$,
\begin{equation*}
\mathbf{A}_{\Sigma}^2+\mathbf{A_{\Gamma}}
<\varepsilon \mathbb{E}(T,C,\bB_1)
\end{equation*}
for $C$ a cornered open book then
\begin{equation*}
\mathbb{E}(T,C,\bB_{\rho}) \leq \frac{1}{2}
\mathbb{E}(T,C,\bB_1).
\end{equation*}
\end{lemma}

\begin{lemma}[Rate of convergence to the  the tangent cone]\label{l:convergencetangentcone}
Let $T, \Gamma_i^j, \Sigma$, $Q_i^j$ be as in Assumption \ref{a:corneredboundary}. There is $\varepsilon>0$, $\alpha>0$ and $\kappa>0$ such that if 
\begin{equation*}
    \max\left\{\mathbb{E}(T,C_p,\bB_4(p)),\kappa^{-1}\mathbf{A_{\Gamma}},\kappa^{-1}\mathbf{A}_{\Sigma}^2\right\}<\varepsilon
\end{equation*} then, for any $q \in \bB_1(p) \cap M$ and $r<1/4$
\begin{equation*}
     \max\left\{\mathbb{E}(T,C_q,\bB_r),\kappa^{-1}\mathbf{A_{\Gamma}}r,\kappa^{-1}\mathbf{A}_{\Sigma}^2r^2\right\} \leq  Cr^{\alpha} \max\left\{\mathbb{E}(T,C_p,\bB_4),\kappa^{-1}\mathbf{A_{\Gamma}},\kappa^{-1}\mathbf{A}_{\Sigma}^2\right\},
\end{equation*}
where $C_q$ denotes the unique tangent cone to $T$ at $q$.
\end{lemma}
\begin{corollary}[Hölder estimate at the corner]\label{c:holdercorner}
Let $T, \Gamma_i^j, \Sigma$, $Q_i^j$ be as in Assumption \ref{a:corneredboundary}.
Then the function $p \rightarrow C_p$ which assigns at points in $M \cap \bB_1$ the tangent cone of $T$ at $p$ is locally Hölder continuous in the Hausdorff distance. This is, for every $\varepsilon>0$ and every $p,q \in M \cap \bB_{1-\varepsilon}$
\begin{equation*}
\textup{dist}_{H}(C_p,C_q) \leq C|p-q|^{\alpha}.
\end{equation*}
\end{corollary}
\begin{remark}
The Hölder continuity of $C_p$ as a function of $p$ in $\bB_1 \cap M$ serves to locally determine the cornered open book when there are multiple admissible cornered open books. This is interesting outside of the multiplicity $1$ case in which there is a unique cornered open book which takes the right boundary data.
\end{remark}
\begin{proof}[Proof of Lemma \ref{l:convergencetangentcone} and Corollary \ref{c:holdercorner}]
If $q$ is required to be $p$ then we need to extend the small second fundamental forms case to the general case as in Theorem 5.8 \cite{ian2024uniqueness}. 

The error of replacing $p$ to $q$ is the error of changing the cornered open book, which we did in detail in Corollary 5.12 \cite{ian2024uniqueness}.

The proof of Corollary \ref{c:holdercorner} holds analogously to Corollary 5.12 \cite{ian2024uniqueness}.
\end{proof}

Now we can prove Theorem \ref{t:decomposition}.
\begin{proof}[Proof of Theorem \ref{t:decomposition}]
We assume $p=0$. Since $C_p$ is the tangent cone, there is a radius $r>0$ such that $\bB_{r}$ is in the hypothesis of the above Lemma \ref{l:convergencetangentcone}. This is 
\begin{equation*}
    \max\left\{\mathbb{E}(T,C_0,\bB_4),\kappa^{-1}\mathbf{A_{\Gamma}},\kappa^{-1}\mathbf{A}_{\Sigma}^2\right\}< \varepsilon
\end{equation*}

Moreover, possibly shrinking the radius, we assume that $M$ is parameterized by bi-lipschitz map whose bi-lipschitz constant is close to $1$. That is, since $L=T_p(M)$, there exists $\psi: L \cap \bB_r \rightarrow \RR^{m+n}$ such that 
\begin{equation*}
M \cap \bB_{2r} = \psi(L \cap \bB_{2r}) \cap \bB_{2r}. 
\end{equation*}

We need to construct a cubical decomposition. We use a variation of the notation introduced for the Lipschitz approximation at the beginning of Section 6.

We define $\mathcal{G}^M_l$ to consist of ''cubes'' $P=\varphi(\hat{P})$, where $\hat{P}$ is a dyadic cube of size $\frac{1}{\sqrt{m-2}}2^{-l}$ in $L$. We define $\lambda P$ as $\psi(\lambda\hat{P})$.

We recall the notation
\begin{equation*}
2R(P):=\left\{p: \mathbf{p}_{M}(p) \in P \: \textup{and} \: 2^{-l-2} \leq \dist(x,M) \leq 2^{-l+1} \right\}.
\end{equation*}
Let $y_P$ be the center of the cube (i.e. $\psi(y_{\hat{P}})$ where $y_{\hat{P}}$ is the center of the original cube).

We notice that there exists a universal constant $\lambda$ such that given $P \in \mathcal{G}_l$ then region is contained in some ball
\begin{equation*}
2R(P) \subset \bB_{\lambda 2^{-l}}( y_P).
\end{equation*}

Thus if $\lambda2^{-\ell(P)}  \leq r/4  $  then $\bB_{\lambda 2^{-l}}(y_P) \subset \bB_{r/4}(y_P) \subset \bB_r$ which implies 
\begin{equation}\label{e:excessbound} \max \left\{\frac{1}{2^{-l(m+2)}}\int_{\bB_{4\lambda 2^{-l}}(y_p)} \dist(x,C_{p_l})^2 d|T|, \kappa^{-1}2^{-l+1}\lambda \mathbf{A_{\Gamma}},\kappa^{-1}\lambda^2 2^{-2l+2}\mathbf{A}_{\Sigma}\right\}\leq C(\lambda,m,n)2^{-l\alpha}  \varepsilon.
\end{equation}

By the separation Lemma \ref{l:splitting}, we obtain that in the region $2R(P)$ the current $T$ decomposes as $\sum_{i=1}^N T_i$, whose supports only intersect at the boundary $\Gamma_i^j$. The regions $2R(P)$ for varying cubes $P$ are chosen to have overlap and additionally they cover $\bB_{r_0}$ for some $r_0<r$. This implies that in $\bB_{r_0}$ the current $T$ can be decomposed as $T=\sum_{i=1}^N T_i$ whose supports only intersect at $\cup_{i,j}\Gamma_i^j \cup M$.

To conclude, we must show the regularity of $T_i$ when the multiplicities are $1$. We can cover $2R(P)$ with balls of size comparable to $2^{-l}$ some centered at the boundary, some at the interior, having some overlap. If $\varepsilon$ is small enough, we can apply Allard's interior regularity \cite{All} for the interior balls and Allard's boundary regularity \cite{AllB} for the balls centered at the boundary (here we are using that the number of balls is a dimensional constant).
\end{proof}

Assume that $N^0=1$, $N^1=N$, $Q_1^0=N$, $Q_1^1=Q_2^1=...=1$.
\begin{definition}\label{d:multivaluednormalmap}
Given $p \in \Gamma_1^0$ let $C_p=\sum_{l=1}^N Q_i\a{H_i}$ be the tangent cone at $p$. We denote $N\Gamma_1^0$ the normal bundle of $\Gamma_1^0$. We define the normal map $\eta:\Gamma_1^0 \rightarrow \mathcal{A}_{Q_1^0}(N\Gamma_1^0)$ as
\begin{equation}
    \eta(p)=\sum_{i=1}^N Q_i \a{\mathbf{p}_{T_p(\Gamma)^{\perp}}(H_i)}.
\end{equation}

We define it analogously at points of $M$ (where it is required to take the value of the unique cornered open book).
\end{definition}

\begin{theorem}\label{t:cornerHöldernormal}
Let $T, \Gamma_i^j,\Sigma,M$ be as Assumption \ref{a:corneredboundary}. Assume that $N^0=1$, $N^1=N$, $Q_1^0=N$, $Q_1^1=Q_2^1=...=1$. Let $p \in \bB_r$. Then there exists a radius $r>0$ such that $\eta$ is Hölder in $\bB_r(p) \cap \Gamma_1^0$ up to the boundary $M$.
\end{theorem}
\begin{proof}
We give a sketch of the argument, which is a modification of the proof of Corollary 5.12 of \cite{ian2024uniqueness}.
By the Theorem \ref{t:decomposition}, $p$ is a regular point and thus it is equivalent to show it for the case that $N^0=1$, $N^1=1$, $Q_1^0=Q_1^1=1$ (i.e., two boundary pieces of multiplicity $1$).

We assume that we are on a neighborhood such that we have the setup of Theorem \ref{t:decomposition}.

Let $H_q$ be the tangent cone at $q$ which is a single quarter plane of multiplicity $1$. By Lemma \ref{l:convergencetangentcone} by restricting to a suitable neighborhood of $p$ we have
\begin{equation*}
\hat{\mathbb{E}}(T,H_q,\bB_{r}(q)) \leq C r^{\alpha}.
\end{equation*}

We use the same cubical decomposition as before. For $x \in 2R(P) \cap \Gamma_0^0$ we have that if we translate and tilt $H_q$ to a new quadrant $H_q^x$ such that $\partial \a{H_q^x}=\a{T_x(\Gamma_0^0)}$ and $H_q^x \subset T_x(\Sigma)$ then 
\begin{equation*}
\mathbb{E}(T,H_q^x,\bB_{2^{-l-4}})+\mathbf{A_{\Gamma}}2^{-l}+\mathbf{A}_{\Sigma}^22^{-2l} \leq C \hat{\mathbb{E}}(T,H_q,\bB_{\lambda2^{-l}}(y_p))\leq C \lambda^{\alpha}2^{-l\alpha}\varepsilon.
\end{equation*}
As a consequence of Allard Boundary regularity \cite{AllB}  we have excess decay and thus if $H_x$ is the tangent cone at $x$
\begin{equation}\label{eq:Hölderatcorner}
\angle (H_x, H_q^x) \leq C \lambda^{\alpha} 2^{-l\alpha}\varepsilon.
\end{equation}
Moreover, our control of the excess we have also implies, up to changing the exponent $\alpha$, the Hölder estimate
\begin{equation}\label{eq:Hölderawayfromcorner}
\left\|\eta \right\|_{C^{\alpha}(2R(P))} \leq C \varepsilon.
\end{equation}

By patching together \eqref{eq:Hölderatcorner} and \eqref{eq:Hölderawayfromcorner} we obtain the desired conclusion.
\end{proof}
\begin{corollary}\label{c:wholeboundaryHölder}
Let $T$,$\Gamma_i^j$, $\Sigma$, $M$ be as in Theorem \ref{t:cornerHöldernormal}. Then $\eta$ is Hölder continuous in $\Gamma_0^0 \cup M$.
\end{corollary}
\begin{proof} By the Theorem \ref{t:cornerHöldernormal}, we have the Hölder continuity locally at $M$. This fact, coupled with the Hölder continuity at a single boundary piece Corollary 5.12 \cite{ian2024uniqueness} implies the Hölder continuity estimate.
\end{proof}

\section{Simon's non concentration estimates}
In this section, we prove Simon's non-concentration estimates - the key technical machinery we require to prove the excess decay theorem.

We can then show 
\begin{lemma}\label{l:simonnonconcentration}
Let $T, \Gamma_i^j,\Sigma,M$ be as Assumption \ref{A:Lipschitzapproximation} for $\varepsilon$ small enough. Let $r=\frac{1}{3\sqrt{m-2}}$
\begin{equation}
\int_{B_r} \frac{|q^{\perp}|^2}{|q|^{m+2}}d|\!|T|\!|(q) \leq C\mathbf{E}(T,C,\bB_4) + C\mathbf{A}_{\Sigma}^2 +C \mathbf{A_{\Gamma}}
\end{equation}
where $q^{\perp}$ is the projection of $q$ into the tangent plane  determined by $\overrightarrow{T}_{q}$.
\end{lemma}
\begin{proof}
The proof is done in the same way as in Theorem 7.1 \cite{ian2024uniqueness}, in that proof we explain how to take care of the boundary terms which are the same that we have here. The error terms appear when using the monotonicity formula with errors Corollary  \ref{c:monotonerrorcorner} and the more general first variation formula from Theorem \ref{t:firstvariationcorneredboundary}. In order to take care of the interior terms we use Theorem 11.2
\cite{de2023fineIII}. 

This involves a careful application of the first variation formula of \ref{t:monotonicityformula} and the monotonicity formula \eqref{c:monotonerrorcorner}. 
For any $\chi$, monotone, non-increasing, such that $\mathbbm{1}_{[0,r]} \leq \chi \leq \mathbbm{1}_{[0,2r)}$. We also define $\Gamma(t):=-\int_t^{\infty} \frac{d}{ds}
(\chi(s)^2)s^mds$.

We bound

\begin{align}
\label{eq:nonconcentration1}\int_{\bB_r} \frac{|q^{\perp}|^2}{|q|^{m+2}}d|\!|T|\!|(q) \leq C\left [ \int \chi^2(|q|)d|\!|T|\!|(q)- \sum_i Q_i\int_{H_i} \chi^2(|q|)d|H_i| \right ]\\\label{eq:nonconcentration2}+ C \int \frac{|\Gamma(|q|)q^{\perp}\cdot\overrightarrow{H}_T|}{|q|^m}d|\!|T|\!|(q)\\\label{eq:nonconcentration3}+C\sum_{i,j}Q_i^j\left[\int_0^{4} \chi(|\rho|)^2 \rho^{-1} \int_{\Gamma_i^j \cap B_{\rho}}(q \cdot \overrightarrow{\eta_{i,j}})d\HH^{m-1}d\rho\right] \\\label{eq:nonconcentration4}+C\sum_{i,j}Q_i^j\left[\int_0^{4} \chi(|\rho|)^2 \rho^{m-1}\int_0^{\rho} t^{-m-1}\int_{\Gamma_{i,j} \cap B_t}(q \cdot \overrightarrow{\eta_{i,j}}) \, d\HH^{m-1}dt d\rho  \right] \end{align}
The term \eqref{eq:nonconcentration2} is dealt analogously to \cite{de2023fineIII}.
The terms \eqref{eq:nonconcentration3}, \eqref{eq:nonconcentration4} are dealt by observing that 
\begin{equation*}
\int_{\Gamma_i^j \cap B_{\rho}}|q\cdot \overrightarrow{\eta_{i,j}}| d\HH^{m-1} = \int_{\Gamma_i^j \cap B_{\rho}}|\mathbf{p}_{T_{q}(\Gamma_i^j)}^{\perp}(q) \cdot \overrightarrow{\eta_{i,j}}| d\HH^{m-1} \lesssim \rho^{m-1} \sup_{q \in \Gamma \cap B_{\rho}} |\mathbf{p}_{T_{q}(\Gamma_i^j)}^{\perp}(q)| \lesssim \rho^{m+1}\mathbf{A_{\Gamma}}.
\end{equation*}

In order to bound the first term we can proceed again as in \cite{de2023fineIII}, but when we use the first variation formula a boundary term gets introduced. 

The vector field that we have to test with is the vector field 
\begin{equation*}
X(q)=\chi(|q|) p_{L}^{\perp}(q).
\end{equation*}
The first variation formula gives us 
\begin{equation*}
\int \textup{div}_{\overrightarrow{T}}(X)=\mathbf{O}(\mathbf{A}_{\Sigma}^2+\mathbf{A_{\Gamma}}).
\end{equation*}

The remaining estimates follow as in \cite{de2023fineIII}, using quantities under control relating the Lipschitz approximation.
\end{proof}
\begin{lemma}
For every $\kappa>0$
\begin{equation}
    \int_{\bB_1}\frac{\dist(q,C)^2}{|q|^{m+2-\kappa}}d|\!|T|\!| \leq C_{\kappa} \int_{\bB_1} \frac{|q^{\perp}|^2}{|q|^{m+2}}d|\!|T|\!|+C_{\kappa}(E(T,C,\bB_4)+\mathbf{A}_{\Sigma}^2+\mathbf{A_{\Gamma}}).
\end{equation}
\end{lemma}
\begin{proof}
This follows again exactly as in Lemma 7.3 \cite{ian2024uniqueness}. For completeness, we again include  the boundary term estimates, reducing the rest of the proof to an adaptation of the proof Lemma 11.6 in \cite{de2023fineIII}.

This lemma follows from a first variation argument on the vector field 

\begin{equation*}
X(q):=\dist^2(q,C) \left(\max(r,|q|)^{-m-2-\kappa}-1\right)_{+}q
\end{equation*}
where $g_{+}(q):=\max(g(q),0).$

The only required modification is the control of the boundary term. 
\begin{equation*}
\int \textup{div}_{\overrightarrow{T}}(X)d|\!|T|\!|= \sum_{i,j} Q_i^j\int_{\Gamma_i^j} (X \cdot \overrightarrow{\eta_{i,j}}) d\HH^{m-1} - \int(X \cdot H_{T})d|\!|T|\!|.
\end{equation*}
We estimate as in the boundary error term in the proof of Lemma \ref{l:simonnonconcentration}
\begin{equation*}
    \int_{\Gamma_i^j} (X \cdot \eta) d\HH^{m-1} \leq \int_{\Gamma_i^j \cap \bB_1} |q|^{-m-\kappa}|\mathbf{p}_{T_{q}(\Gamma_i^j)^{\perp}}(q)\cdot \eta_{i,j}| d\HH^{m-1} \leq C_{\kappa} \mathbf{A_{\Gamma}}.
\end{equation*}
\end{proof}

\begin{theorem}\label{t:nonconcentration}
Let $T, \Gamma, \Sigma, C$ be as in Assumption \ref{A:Lipschitzapproximation}. For every $\overline{\alpha}>0$
\begin{enumerate}
    \item For every cube $L$ in the cubical decomposition from the previous chapter
\begin{equation*}
\mathbf{E}(P) \lesssim_{\overline{\alpha}} \ell(P)^{-\overline{\alpha}} \left(\mathbf{E}(T,C,\bB_4)+\mathbf{A_{\Gamma}}+\mathbf{A}_{\Sigma}^2\right)
\end{equation*}
\item 
    For every $\sigma$ and $r \leq \frac{1}{6\sqrt{m-2}}$:
\begin{equation*}
  \frac{1}{r^{m+2}}  \int_{\bB_r \cap B_{\sigma r}(L)}\dist(q,C)^2 d|\!|T|\!|  \leq C_{\delta}\sigma^{4-\overline{\alpha}}  \left(\mathbf{E}(T,C,\bB_4)+\mathbf{A_{\Gamma}}+\mathbf{A}_{\Sigma}^2\right).
\end{equation*}
\item Moreover, 
 if $\bB_r \setminus B_{\sigma r}(L) \subset R^{out}$ then 
\begin{align*}
\frac{1}{r^{m+2}}\int_{\bB_r} \dist(q,C)^2 d|\!T|\!| \leq C(\delta,Q,m,n,\overline{n})\left( \mathbf{E}(T,C,\bB_4) + \mathbf{A_{\Gamma}}+\mathbf{A}_{\Sigma}^2 \right) \sigma^{4-\overline{\alpha}} \\+ \frac{1}{r^{m+2}}\sum_{i}\int_{\bB_r \cap H_i \setminus B_{\sigma r}(L)} |u_i|^2 \\+C(Q,m,n,\overline{n})\left( \mathbf{E}(T,C,\bB_4)+r\mathbf{A_{\Gamma}}+r^2\mathbf{A}_{\Sigma}^2 \right)\left(\mathbf{E}(T,C,\bB_4)\sigma^{-m-2}+r\mathbf{A_{\Gamma}}+r^2\mathbf{A}_{\Sigma}^2\right)^{\gamma}
\end{align*}
\end{enumerate}
\end{theorem}
\begin{proof}
The proof is the same as Theorem 7.4 \cite{ian2024uniqueness}. 
\end{proof}

\section{Proof of excess decay}
In this section we prove the excess decay lemma, introduced in Section 7, from which we deduce the Hölder continuity of the multivalued normal and the regularity at multiplicity $1$ corner points.
\begin{assumption}[Blowup Sequence]\label{a:blowupsequence}
A sequence of currents $T_k$ is said to be a blowup sequence if they satisfy Assumption \ref{A:general}, $T_{0}(L)=\RR^{m-2}\times \left\{0 \right\}$, $C_k=\sum_{i=1}^N Q_i\a{H_{i,k}}$ cornered open books (with the same number of sheets and labeling) and:
\begin{equation*}
\mathbb{E}(T_k,C_k,\bB_8)\leq \frac{1}{k} \; \textup{and} \; \mathbf{A_{\Gamma}}^2+\mathbf{A}_{\Sigma}^2\leq \frac{1}{k}\mathbb{E}(T_k,C_k,\bB_8).
\end{equation*}
\end{assumption}
We wish to take a Dir-minimizing blowup out of a blowup sequence of currents. Up to a subsequence, $H_{i,k} \rightarrow H_i$. If $C:=\sum Q_i \a{H_i}$ with $\partial\a{H_i}=\a{V_{l(i)}^0}-\a{V_{l(i)}^1}$ then $C_k \rightarrow C$. We also have $T_k \rightarrow C$, $\Sigma_k \rightarrow \RR^{m+n}$, $\Gamma_{i,k}^j \rightarrow V_i^j$. The cornered boundaries satisfy $M_k \rightarrow L$.

If $k$ is large enough, $T_k$ satisfy the Assumption \ref{A:Lipschitzapproximation}. The excess bound is obviously satisfied so it only remains to check the mass bound. This is a consequence of the monotonicity formula and the fact that $T_k  \rightarrow C$ with $C$ an open book of multiplicity $Q$.

We take the Lipschitz approximations $u_i^k$ whose domains of definition are the outer region which are subdomains of $\left(\bB_1 \setminus B_{1/k}(L)\right)\cap D_{i,k} $. We need them to be defined on $\bB_1 \setminus B_{1/k}(L) \cap H_i $. This is easy to achieve by applying a rotation we define for $i,k$ a rotation $\sigma_{i,k}$ on $\RR^{m+n}$ that takes $H_{i,k}$ to $H_i$, while fixing $\RR^{m-1}\times \left\{0\right\}$. Clearly, $\sigma_{i,k}\rightarrow_{k} \textup{Id}$. We define
\begin{equation*}
\bar{u}_i^k:=\frac{\mathbf{p}_{H_i^{\perp}} \circ \sigma_{i,k} \circ u_i^k\circ\sigma_{i,k}^{-1}}{\sqrt{\mathbb{E}(T_k,C_k,\bB_8)}}.
\end{equation*}

Up to a subsequence $\overline{u}_i^k$ converges weakly to a function $\overline{u}_i$. We will be able to estimate the functions $\overline{u}_i$ at the scale $r=\frac{1}{12\sqrt{m-1}}$ since we have the non-concentration estimate at scale $\frac{1}{6\sqrt{m-1}}$. The functions $\overline{u}_i$ will be the Dir-minimizing blowup of the blowup sequence. 

\begin{lemma}[Dir-minimizing blowup]
Let $r=\frac{1}{12\sqrt{m-1}}$. We list properties of the sequence $\overline{u}_i^k$ and their weak limit $\overline{u}_i$:
\begin{itemize}
\item The functions $\overline{u}_i^k$ have equibounded Dirichlet energy.
\item $\overline{u}_i^k$ converge strongly in $W^{1,2}(\bB_r^+,A_{Q_i}(\RR^n))$ (up to a subsequence) to a function $\overline{u}_i$ in $W^{1,2}(\bB_r^+,A_{Q_i}(\RR^n))$.
\item  The functions $\overline{u}_i$ are Dir-minimizing on $\bB_r^+$ with zero boundary value along $V_{l(i)}^0 \cup V_{l(i)}^1$.
\end{itemize}
\end{lemma}
\begin{proof}
The proof is the same as Lemma 9.3 of \cite{ian2024uniqueness}. The only new aspect is that the zero boundary data holds through the classical boundary parts, which follows from the fact that $\frac{\mathbf{A_{\Gamma}}}{\mathbb{E}(T_k,C_k,\bB_8)} \rightarrow 0.$
\end{proof} 

\begin{proof}[Proof of Lemma \ref{l:excessdecay}]
We rescale the ball $\bB_1$ to $\bB_8$ and get a sequence of currents as in Assumption \ref{a:blowupsequence}. 

We prove the lemma by contradiction. We claim it suffices to show that
\begin{equation}\label{eq:excessdecayfinal}
\limsup_{r \rightarrow 0} \limsup_{k \rightarrow \infty}\frac{\mathbb{E}(T_k,C_k,\bB_r) }{\mathbb{E}(T_k,C_k,\bB_8)}=0.
\end{equation}
If the lemma was false there exists $\theta$ and a sequence of radii $r_{k_l} \rightarrow 0$  such that
\begin{equation*}
\frac{\mathbb{E}(T_{k_l},C_{k_l},\bB_{r_{k_l}}) }{\mathbb{E}(T_{k_l},C_{k_l},\bB_8)}>\theta.
\end{equation*}
This would contradict equation \eqref{eq:excessdecayfinal} and conclude the proof of the lemma.

It remains to prove \eqref{eq:excessdecayfinal}.

We have that 
\begin{equation*}
\mathbf{E}(T_k,C_k,\bB_4) \leq C \mathbb{E}(T_k,C_k,\bB_8). 
\end{equation*}

By the non-concentration estimate for every $\sigma$ if $r \leq \frac{1}{6\sqrt{m-2}}$ then 
\begin{align*}
\frac{1}{r^{m+2}}\int_{\bB_r} \dist(q,C_k)^2 d|T_k| &\leq C(\delta,Q,m,n,\overline{n})\left(\mathbf{E}(T_k,C_k,\bB_4) + \mathbf{A_{\Gamma}}+\mathbf{A}_{\Sigma}^2 \right) \sigma^{4-\delta} \\&+ \frac{1}{r^{m+2}}\sum_{i}\int_{\bB_r \cap H_i \setminus B_{\sigma r}(L)} |u_{i,k}|^2 \\+C(Q,m,n,\overline{n})\left( \mathbf{E}(T_k,C_k,\bB_4)+r\mathbf{A_{\Gamma}}+r^2\mathbf{A}_{\Sigma}^2 \right)&\left(\mathbf{E}(T_k,C_k,\bB_4)\sigma^{-m-2}+r\mathbf{A_{\Gamma}}+r^2\mathbf{A}_{\Sigma}^2\right)^{\gamma}
\end{align*}
We divide by $\mathbb{E}(T_k,C_k,\bB_8)$. We take $\limsup k \rightarrow 0$. Then, as $\sigma \rightarrow 0$, the last term vanishes since it is superlinear with respect to the excess and the first term will disappear because we took $\sigma$ to zero. 
Therefore, for $r \leq \frac{1}{12 \sqrt{m-2}}$
\begin{equation*}
\limsup_{k \rightarrow \infty} \frac{\mathbf{E}(T_k,C_k,\bB_r)}{\mathbb{E}(T_k,C_k,\bB_8)}\leq \frac{1}{r^{m+2}}\sum_{i} \int_{\bB_r}|\overline{u}_i|^2.
\end{equation*}
We conclude that 
\begin{equation*}
\lim_{r \rightarrow 0} \limsup_{k \rightarrow 0}\frac{\mathbf{E}(T_k,C_k,\bB_r)}{\mathbb{E}(T_k,C_k,\bB_8)}=0.
\end{equation*}
Thus, by the reverse excess Lemma \ref{l:reverseexcessbound}, we conclude
\begin{equation*}
\lim_{r \rightarrow 0} \limsup_{k \rightarrow 0}\frac{\mathbb{E}(T_k,C_k,\bB_r)}{\mathbb{E}(T_k,C_k,\bB_8)}=0
\end{equation*}
which gives a contradiction and proves the excess decay lemma.
\end{proof}
\part{Construction of the example}
\section{Classification of certain cornered tangent cones under a convexity assumption}
In this section, we will classify some cornered tangent cones which are necessary to provide our example. The tangent cones we classify are those we will obtain in corners under a convex barrier hypothesis similar to Assumption \ref{convexbarrier}.

\begin{theorem}\label{thm:corneredconvex} Let $C$ satisfy the following hypothesis in $\RR^{3+n}$:
\begin{enumerate}
    \item We consider the half-planes $V^0=\RR e_3 \oplus \RR^+ e_2$, $V_i^1= \RR e_3 \oplus \RR^+ (-1)^i e_1$ which we orient so that $\partial \a{V^0}=-\partial\a{V_i^1}$.   
    \item Assume that $C$ is an area-minimizing cone such that $\partial C=2\a{V^0}+\sum_{i=1}^2 \a{V_i^1}$.
    \item Assume that the support of the cone satisfies
    \begin{equation}\label{e:convexbarriercorner}
\textup{spt}(C) \subset  W:=\left\{p=(x_1,x_2,x_3,x_4,y)\in (\RR)^4 \times \RR^{n-1} :|x_2^{-}| + |y| \leq C_0 x_{4} \right\}
\end{equation}
where $x_2^{-}= -\max\left\{0,x_2 \right\}$ and $C_0$ is a positive constant.
\end{enumerate}
Then the cone $C$ is a cornered open book of the form shown Figure \ref{img:nicecorneredopenbook}:
\begin{equation*}
C=\a{\RR^{+}e_1\oplus\RR^{+}e_2 \oplus \RR e_3}+\a{\RR^{+}(-e_1) \oplus \RR^{+}e_2 \oplus \RR e_3}
\end{equation*}
when the orientations of the quarter spaces match appropriately the orientation of the boundary data.
\end{theorem}

We depict the boundary of the cone we consider in the following picture: 

\begin{figure}[H] 
    \centering 
    \includegraphics[width=0.5\linewidth]{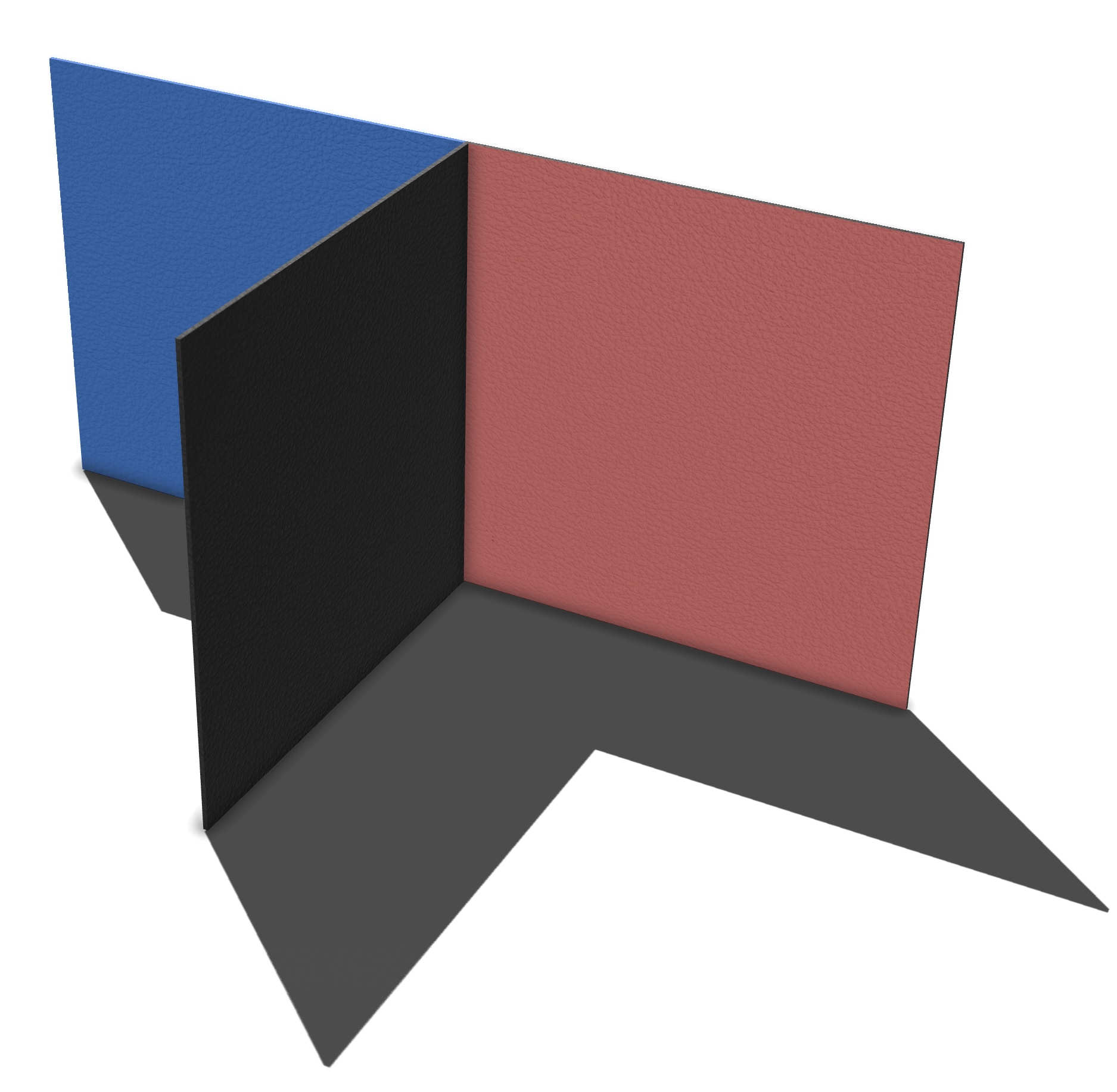}
    \caption{\small Boundary of the Cone}
    \vspace{-10pt}
    \label{img:boundaryofcone}
\end{figure}
The half-plane $V^0$ is depicted in black, the half-planes $V_i^1$ are depicted in red and blue respectively. The theorem says that a cone $C$ with the boundary as in Figure \ref{img:boundaryofcone} is a cornered open book of the type depicted in Figure \ref{img:nicecorneredopenbook}.

We remind of the following theorem originally due to Allard in \cite{AllB} which we expanded on in Section 4.3 \cite{ian2024uniqueness} and applied it to a higher multiplicity setting.

\begin{theorem}\label{t:allardconvexvarifoldslemma} Fix a $2$-dimensional plane $\pi_0 \subset \RR^{m+n}$ and assume that:
\begin{enumerate}
    \item $\mathcal{V}$ is an $m$-dimensional varifold which is stationary on $\RR^{m+n} \setminus \pi_0^{\perp}$.
    \item $\Theta(\mathcal{V},x) \geq c_0 >0$ for some positive numbers $c_0$ and for $|\!|
    \mathcal{V}|\!|$-almost every $x \in \textup{spt}(V) \setminus \pi_0^{\perp}$.
    \item $\mathcal{V}$ is a cone (i.e. $\left(\lambda_{0,r}\right)_{\#}\mathcal{V}=\mathcal{V}$).
    \item The orthogonal projection of $\textup{spt}(\mathcal{V})$ onto $\pi_0$ is contained on a half-space). This means there is a vector $e \in \pi_0 \cap \mathbb{S}^{m+n-1}$ such that 
    \begin{equation*}\label{e:convexbarriervarifold}
        \mathbf{p}_{\pi_0}(\textup{spt}(\mathcal{V}))\subset \{x:0 \leq \langle x,e \rangle \}.
    \end{equation*}
\end{enumerate}
\begin{itemize}
\item Then $\exists$ finitely many $v_1,..., v_k \in \pi_0 \cap \mathbb{S}^{m+n-1}$ such that 
\begin{equation*}
 \left( \bB_1 \cap \textup{spt}(
 \mathcal{V}) \right) \setminus B_{\rho}(\pi_0^{\perp})  \subseteq \bigcup_{i=1}^{k(\rho)} \pi_0^ {\perp} + \RR^+ v_i
\end{equation*}
for every $\rho>0$. 
\item We denote $(x,y)=z$, $x^{\perp}=(-x_2,x_1)$ if $x=\mathbf{p}_{\pi_0}(z)$.  For $\mathcal{V}$ a.e. $(z,\pi)$ where $z \in \RR^{m+n} \setminus{\pi_0^{\perp}} $ we have $\mathbf{p}_{\pi}(x^{\perp})=0.$
\end{itemize}
\end{theorem}

Assume $\pi_0= \{(x_1,x_2,0,...,0) \in \RR^{m+n} \}$ and we denote $z=(x,y) \in \pi_0 \times \pi_0^{\perp}$. The following lemma implies Theorem \ref{t:allardconvexvarifoldslemma}:
\begin{lemma}\label{l:allardwinding}Assume $\mathcal{V}$ is a varifold which is stationary on $\RR^{m+n} \setminus \pi_0^{\perp}$ and is a cone (conditions (1) and (3) of Theorem \ref{t:allardconvexvarifoldslemma}). Consider the following Radon measure $\mu$ on $\mathbb{S}^1:=\{x \in \pi_0: |x|=1\}$
\begin{equation} \label{e:allardmeasure}
\int \varphi d\mu =\int_{\bB_1 \setminus \pi_0^{\perp}} \varphi \left( \frac{x}{|x|}\right) |x|^{-2} |\mathbf{p}_{\pi}(x^{\perp})|^2d\mathcal{V}(z,\pi)
\end{equation}
    where $(x,y)=z$, $x^{\perp}=(-x_2,x_1)$. Then $\mu=c \HH^1$ for some $c \in \RR$. 
    
    In addition, if $\mu=0$ and we have conditions (2) and (4) in Theorem \ref{t:allardconvexvarifoldslemma} then:
    \vspace*{-0.2cm}
\begin{enumerate}
\item $\mathbf{p}_{\pi}(x^{\perp})=0$ for $\mathcal{V}$ a.e. $(z,\pi)$ where $z \in \RR^{m+n} \setminus{\pi_0^{\perp}}$.
\item The set $\mathbf{p}_{\pi_0}(\textup{spt}(\mathcal{V}) \setminus B_{\rho}(\pi_0^{\perp})) \cap S^{1}$ is finite and it has at most $\frac{2^m \Theta(\mathcal{V},0)}{\rho^m c_0}$ points.
\end{enumerate}  
\end{lemma}
The theorem above is a direct consequence of this lemma and we reproved it in \cite{ian2024uniqueness}.
\begin{proof}[Proof of Theorem \ref{thm:corneredconvex}]
We denote $V^1:=V_0^1 \cup V_1^1= \RR e_1 \oplus \RR e_3.$
We can apply $n-1$ times the above Lemma \ref{l:allardwinding}, each time taking $\pi_0$ as $\RR e_4 \oplus \RR e_{4+k}$  , for every $1 \leq k \leq n-1$.

We define the vector fields $v_k:= -x_4 e_{4+k}+x_{4+k}e_4.$
Thus we get that for $|\!|C|\!|$- almost every $p \in \textup{spt}(C) \cap \left\{x_4>0 \right\}$ we have
\begin{equation*}
    \mathbf{p}_{\overrightarrow{C}}(v_k)=0 \; \forall 1 \leq k \leq n-1.
\end{equation*}

We will show in Lemma \ref{l:allardwindingsingular} that a variant of Lemma \ref{l:allardwinding} holds for $\pi_0=\RR e_2 \oplus \RR e_4.$. This implies $v_0:=x_2e_4-x_4e_2$ then 
\begin{equation}\label{eq:appsingularallardwinding}
\mathbf{p}_{\overrightarrow{C}}(v_0)=0 \;\textup{for}\; |\!|C|\!|-a.e.
\end{equation}

This allows us to conclude that if we define $\sigma:\RR^{3+n} \setminus V^1 \rightarrow {V^1}^{\perp}$ as
\begin{equation}
\sigma(x):=\frac{\mathbf{p}_{{V^1}^{\perp}}(x)}{|\mathbf{p}_{{V^1}^{\perp}}(x)|}
\end{equation}
then $|\!|C|\!|$ almost everywhere $D_{C}(\sigma) \equiv 0.$

We conclude assuming Lemma \ref{l:allardwindingsingular} and thus its consequence, that $|\!|C|\!|$ almost everywhere $D_{C}(\sigma) \equiv 0.$ By interior regularity theory due to Almgren \cite{Alm} and  De Lellis - Spadaro \cite{DS3,DS4,DS5}, $C$ is regular in the interior up to a set of codimension $2$. If $q$ is an interior regular point for $C$ and $P_q:=\sigma^{-1}(\sigma(q))$ is independent of the spine $L$, and moreover by constancy 
\begin{equation*}
|C|(P_q \cap \bB_1) \geq \frac{\omega_{m}}{4}.
\end{equation*}

Given such a plane $P_q$, it must part of it must lie in $\left \{x_4 \geq 0\right\}$ since $C$ assigns mass to it. There are two options: either $P_q \subset \left\{x_4=0 \right\}$ holds or it doesn't. 

We define $H_0= \RR e_1 \oplus \RR^{+} e_2 \oplus \RR e_3$. We split $H_0$ into $H_0^+=H_0 \cap \left\{x_1 \geq 0\right\}$ and $H_0^-=H_0 \cap \left\{x_1 \leq 0\right\}$.

\begin{itemize}
\item If $|P_q|(\left\{x_4=0\right\})=0$ then 
 we consider $H_p:= P_q \cap \left \{x_4 \geq 0\right\}$
must be a half-plane with boundary $\a{\RR e_1 \oplus \RR e_3}$. We have that $|C|(P_q \setminus H_p)=0$ and $|C|(H_p \cap \bB_1) \geq \frac{\omega_{m}}{2}$.
\item If $P_q \subset \left \{x_4=0\right\}$ (which is equivalent to $|P_q|(\left\{ x_4>0 \right\})=0$ since $V^1 \subset P_q)$, we must have that $q \in \RR e_1 \oplus \RR^{+}e_2 \oplus \RR e_3$.
In this case we must have by constancy that the multiplicity is constantly nonzero either in $H_0^+$ or in $H_0^-$ and thus $|C|(P_q \cap \bB_1) \geq \frac{\omega_{m}}{4}.$  We orient $\a{H_0^{\pm}}$ such that $\partial \a{H_0^+}=\a{V^0}+\a{V_2^1}$ and $\partial \a{H_0^{-}}=\a{V^{-}}+\a{V_1^1}$.

\end{itemize}

Assume that there are at least $N$ distinct planes $P_i$ containing a regular point. Then
\begin{equation*}
|C|(\bB_1) \geq \sum_{i=1}^N |C|(P_i \cap \bB_1) \geq \frac{N \omega_m}{4}
\end{equation*}
this implies there are only finitely many such planes.

We must have that for $Q_i$ integers
\begin{equation*}
C \res \left\{x_4>0\right\}=\sum_{i=1}^N Q_i \a{H_k}.
\end{equation*}
Since $\textup{spt}(C)\cap \left\{x_4=0\right\} \subset H_0$, by constancy we obtain for $Q^{\pm}$ integers
\begin{equation*}
C= \sum_{\pm} Q^{\pm} \a{H_0^{\pm}} +\sum_{i=1}^N Q_i \a{H_k}.
\end{equation*}

Up to reversing the orientations of $\a{H_i}$ we can assume that $Q_i>0$ for every $1 \leq i \leq N$. Given that $\textup{spt}(\partial\a{H_i}) \subset V^1$, we will show that $\partial\a{H_i}=\partial\a{H_j}$ for all $1 \leq i<j \leq N$.
Assume that instead for some $1 \leq i < j \leq N$, we have $\partial \a{H_i} +\partial \a{H_j}=0.$

By taking $S=\a{H_i}+\a{H_j}$, we can write $C$ as $C=S+T$ where $|\!|C|\!|=|\!|S|\!|+|\!|T|\!|$. This implies that $S$ is area minimizing without boundary. The only option for this to hold in this setting given that the planes have a codimension $1$ spine is if $H_ \cup H_j$ is a full plane. This is impossible since both of them are contained in $\left\{x_4 \geq 0\right\}.$

In order to take the boundary data the following three equations must be satisfied
\begin{enumerate}
    \item $Q^{+}+Q^{-}=2$
    \item $Q^{+}+\sum_{i=1}^N Q_i=1$
    \item $Q^{-}+\sum_{i=1}^N Q_i=1$.
    \end{enumerate}
It follows that $Q^{+}=Q^{-}=1$ and $\sum Q_i=0$. Since all $Q_i$ have the same sign, they must be zero.

\textbf{Allard's lemma for the plane $\RR e_2 \oplus \RR e_4$}

In order to conclude we must show that we can apply a variation of Lemma \ref{l:allardwinding} to the present case with the plane $\pi_0=\langle e_2,e_{4} \rangle$. We will show that an appropriate version of the lemma holds in the hypothesis that we have. Inside the proofs of these lemmas we will adopt the notation $z=(x,y) \in \pi_0$ and $x^{\perp}=(-x_4,x_2).$

\begin{lemma}\label{l:allardwindingsingular} We consider $\pi_0=\langle e_2,e_4 \rangle$ and $\mathcal{V}$ the varifold induced by $C$. Then the measure $\mu$ defined by \eqref{e:allardmeasure} satisfies $\mu=c \HH^1$, which implies in the present context $\mu \equiv 0$. Thus, $\mathbf{p}_{\pi}(x^{\perp})=0$ for $\mathcal{V}$ a.e. $(z,\pi)$ where $z \in   \RR^{m+n} \setminus \left(\pi_0^{\perp} \cup \textup{spt}(\partial C) \right)$. 
\end{lemma}
\begin{proof}
 To prove that $\mu=c\HH^1$ we must show that if $\varphi \in C(S^1)$ and has zero average, i.e. $\int_{\mathbb{S}^1} \varphi d\HH^1=0$, then 
 \begin{equation*}
     \int \varphi d\mu=0.
 \end{equation*}

We consider $\varphi$ as a $2\pi$ periodic function on $\RR$ and we take $\psi(t):=\int_0^t \varphi(\tau)d\tau$ a primitive of $\varphi$ on $\RR$.

We consider a vector field of the form 
\begin{equation*}
X(z)=\alpha(|z|)\beta(|x|)\psi(x)x^{\perp}.
\end{equation*}
Where $\alpha, \beta \in C^{\infty}(\RR)$ are defined with the following properties:
$\alpha \equiv 1$ for $t \leq 1-\delta$, $\alpha \equiv 0$ for $t \geq 1 - \delta/2$, and it is decreasing.
$\beta \equiv 1$ for $t \geq \delta$, $\beta \equiv 0$ for $t<\delta/2$, and it is increasing with $\|{\beta}'\|_{C^0} \lesssim \delta^{-1}$.

We will have that 

\begin{equation*}
\int \textup{div}_{\pi}X(z)d\mathcal{V}(z,\pi)=2\int_{V^0} \int X \cdot \eta^0 d\HH^{2}.
\end{equation*}

We take $\delta \rightarrow 0$ which implies by the arguments in Lemma 4.12 \cite{ian2024uniqueness} that
\begin{equation*}
\int \varphi d\mu= \lim_{\delta \rightarrow 0} \int \textup{div}_{\pi}Xd\mathcal{V}(z,\pi)= \psi(0,1) 2\int_{V^0 \cap B_1(0)} \left(x^{\perp}\cdot \eta^0 \right)d\HH^{2}.
\end{equation*}
Since the left hand side of the equality is independent of the choice of $\psi$, the primitive of $\varphi$, then the right hand side must be independent of the choice $\psi$. This implies that the right hand side equals zero. 

Alternatively, the vanishing of the right-hand side can be seen
directly by testing with the vector field
\begin{equation*}
Y(z)=\alpha(|z|)\beta(|x|)x^{\perp}
\end{equation*}
with $\alpha$ and $\beta$ as above and taking the limit $\delta \rightarrow 0$.

Since 
\begin{equation*}
    \textup{spt}(\mu) \subset \left\{|x_2^{-}| \leq C x_4 \right\}
\end{equation*}
then $\mu \equiv 0$ and thus $\mathbf{p}_{\pi}(x^{\perp})=0$ for $\mathcal{V}$ a.e. $(z,\pi) \in \RR^{m+n} \setminus \pi_0^{\perp}.$
\end{proof}
Lemma \ref{l:allardwindingsingular} implies \ref{eq:appsingularallardwinding}, which in turn implies $D_{C}(\sigma)\equiv 0$ $|\!|C|\!|$-almost everywhere, allowing us to conclude.
\end{proof}
\section{Construction}
In this section, we construct a general family of examples of area-minimizing currents with essential one-sided boundary singularities. We start the section by providing an example of a specific boundary which is real analytic to which the theorems will apply, which allows us to conclude Theorem \ref{t:exampleintro} from Theorem \ref{T:generalexample}. 

\subsection{Concrete example in $\RR^5$}
The lower part of the boundary $\Gamma^0$ is defined to be the $2$-dimensional lower hemisphere
\begin{equation*}
\Gamma^0:=\left\{(x_1,x_2,x_3,0,0) \in \mathbb{S}^4: x_3<0 \right\}.
\end{equation*}

The upper part of the boundary $\Gamma^1$, is described by a two-valued map 
$\Phi^1:\mathbb{S}^1 \times [0,1] \rightarrow \mathcal{A}_2(\mathbb{S}^4) \subset \mathcal{A}_2(\RR^{5})$. 

We consider $\mathbb{S}^1 \subset \mathbb{C}$ so that the complex multiplication and thus the $2$-valued square root makes sense.

Let

\begin{equation*}
\Phi^1(z,t):=\sum_{w^2=z} \a{\left(\sqrt{1-t^2-t^4}z,t^2,tw)\right)}.
\end{equation*}

We define 
\begin{equation*}
\Gamma^1:=\left(\Phi^{1}\right)_{\#}(\mathbb{S}^1 \times [0,1]).
\end{equation*}

Let $M:=\left\{(z,0,0) \in \mathbb{C} \times \RR \times \mathbb{C}: |z|=1 \right\}$.
We orient $\Gamma^0$, $\Gamma^1$ and $M$ so that $2\partial\a{\Gamma^0}=2\a{M}=-\partial \a{\Gamma^1}.$

\begin{theorem}\label{T:exampledetail} Let $T$ be a $3$-dimensional area-minimizing current such that $\partial T=2\a{\Gamma^0}+\a{\Gamma^1}$ with $\Gamma^0$ and $\Gamma^1$, being the $2$d integral currents defined as above. Then $T$ has an essential boundary one-sided singularity at some interior point of the boundary surface $\Gamma^0$.
\end{theorem}

By localizing to a small ball centered at an essential singular point, we obtain the following Theorem:
\begin{theorem}\label{T:example}
There exist a $3$-dimensional area-minimizing current $T$ in $\bB_1 \subset \mathbb{R}^5$ and a real analytic surface  $\Gamma$ such that $\partial T = 2\a{\Gamma}$, 
$\Theta(T,p) = 1$ for every $p \in \Gamma \cap \bB_1 $, and $0$ is an essential one-sided boundary singularity.
\end{theorem}

\subsection{A class of general examples}
\begin{assumption}[Density hypothesis]\label{a:densityhypothesis}
We say that two $m \geq 3$ dimensional integral currents satisfy $\Gamma^0$ and $\Gamma^1$  with $\partial \a{\Gamma^0}=\a{M}$, $\partial \a{\Gamma^1}=-2\a{M}$ and $\textup{spt}(\Gamma^0)\cap\textup{spt}(\Gamma^1)=M$ satisfy the density hypothesis if:
\begin{itemize}
\item There is a $C^2$ open set $\Omega$ which has uniformly convex boundary and $M \subset \partial \Omega$ is a smooth closed $m-2$ dimensional submanifold.
\item $\textup{spt}(\Gamma^0) \subset \partial \Omega$ is a $C^2$ surface with boundary $M$. 
\item $\textup{spt}(\Gamma^1)$ is a smooth surface in a tubular neighborhood $U$ of $M$.
\item $\textup{spt}(\Gamma^0)$ and $\textup{spt}(\Gamma^1)$ only intersect at $M$ and they have orthogonal normals along $M$. 
\end{itemize}
an area-minimizing current $T$ with $\partial T=2\a{\Gamma^0}+\a{\Gamma^1}$ is said to satisfy the density hypothesis if $\Gamma^0, \Gamma^1$ satisfy the density hypothesis.
\end{assumption}
The convexity implies that $\textup{spt}(T) \subset \overline{\Omega}$. The convexity allows us to say that if $p \in \Gamma^0 \setminus M$ then $\Theta(T,p)=1$, if $p \in U \cap \textup{spt}(\Gamma^1) \setminus M$ then $\Theta(T,p)=1/2$ and if $p \in M$ then $\Theta(T,p)=1/2.$ This allows us to use the developed regularity theory. 

\begin{assumption}[Additional Topological hypothesis]\label{a:topologicalhypothesis}
 We say that $\Gamma^0$ and $\Gamma^1$, two $2$ dimensional integral currents, satisfy the additional topological hypothesis if it satisfies the density hypothesis \ref{a:densityhypothesis} for $m=3$ and
\begin{itemize}
\item $\Gamma^0$ is diffeomorphic to a disk with boundary $M$ (which is a closed simple curve).
\item  
In a tubular neighborhood $U$ of $M$, the surface $\textup{spt}(\Gamma^1)$ is non-orientable: $U \cap \textup{spt}(\Gamma^1)$ is a $C^2$ surface diffeomorphic to a Möbius strip twisting once around $M$, and its orientation reverses across $M$.
\end{itemize}
A $3$ dimensional area-minimizing current $T$ with $\partial T=2\a{\Gamma^0}+\a{\Gamma^1}$ is said to satisfy the additional topological hypothesis if $\Gamma^0, \Gamma^1$ satisfy the additional topological hypothesis.
\end{assumption}
We show the following theorem, which provides a large family of examples of essential one-sided boundary singularities, by splitting it into two parts.
\begin{theorem}\label{T:generalexample}
Let $\Omega$ be a $C^2$ uniformly convex open set. 
\begin{itemize}
\item Let $M \subset \partial \Omega$ be a $C^2$ simple closed curve.
\item Let $\Gamma^0$ be a $C^2$ surface with $\partial \a{\Gamma^0}=M$ diffeomorphic to a $2$ dimensional disk.
\item Let $\Gamma^1 \subset \partial \Omega$ be an integral $2$ dimensional current. We assume that $M \subset \textup{spt}(\Gamma^1)$ and in a neighborhood $U$ of $M$, $U \cap \textup{spt}(\Gamma^1)$ is a $C^2$ surface diffeomorphic to a Möbius strip which twists once around $M$. Indeed, $U \cap \textup{spt}(\Gamma^1) \setminus M$ will be connected. We assign multiplicity $1$ to $\textup{spt}(\Gamma^1)$ in $U$. We suppose that $\Gamma^1$ is oriented such that $\partial \a{\Gamma^1}=-2\a{M}$. The surface $U \cap\Gamma^1$ is a smooth surface whose orientation switches across $M$. 
\item Further assume $\Gamma^0$ and $\Gamma^1$ meet perpendicularly at $M$ and $\Gamma^0 \cap \Gamma^1=M.$
\end{itemize}
Let $T$ be an area-minimizing current with $\partial T=2\a{\Gamma^0}+\a{\Gamma^1}$. Then $T$ is regular in a neighborhood of $M$ and has an essential boundary singularity at some point of $\Gamma^0 \setminus M$. 
\end{theorem}

The first part is the control on the density and the regularity theory:

\begin{theorem}\label{t:densityexample}
Let $T, \Omega, \Gamma^0, \Gamma^1$ be as in the density hypothesis Assumption \ref{a:densityhypothesis}. Then
\begin{enumerate}
\item $\forall p \in \Gamma^0 \setminus M, \; \Theta(T,p)=1$.
\item $\forall p \in \Gamma^1 \cap U \setminus M, \; \Theta(T,p)=1/2$.
\item $\forall p \in M, \; \Theta(T,p)=1/2$ and thus $T$ is regular at $M$.
\end{enumerate}

The tangent cone is unique at all the above points, and the multivalued normal map at $\Gamma^0$  is Hölder continuous up to the boundary $M$ (where it agrees with the cornered open book). Moreover, $T$ is regular in a neighborhood of $M$. \end{theorem}
\begin{remark}
The density $\Theta(T,p)$ is defined differently at interior points $p$, at classical boundary points $p \in \Gamma^0 \setminus M$ or $p \in \Gamma^1 \cap U \setminus M$ and at the corner points $M$. It is upper semi-continuous in each of these pieces. 
\end{remark}
We start with a quick discussion of why we consider the convex barrier property. The convex barrier property that we considered in \cite{ian2024uniqueness} and was considered in \cite{delellis2021allardtype} is the following:
\begin{assumption}[Convex Barrier] \label{convexbarrier2} Let $\Omega \subset \RR^{m+n}$ be a domain such that $\partial \Omega$ is a $C^2$ uniformly convex submanifold of $\RR^{m+n}$. We say that $\sum_{i=1}^N Q_i \a{\Gamma_i}$ has a convex barrier if $Q_i$ are positive integers and $\Gamma_i \subset \partial \Omega$ are disjoint $C^2$ closed oriented submanifolds of $\partial \Omega$. In this setting we consider $T$ an area-minimizing current with $\partial T=\sum_i Q_i \a{\Gamma_i}.$ 
\end{assumption}
In order to get to the local information we define the wedge regions
\begin{definition}
Given an $(m-1)$-dimensional plane $V \subset \mathbb{R}^{m+n}$ we denote by $\mathbf{p}_V$ the orthogonal projection onto $V$. Given additionally a unit vector $\nu$ normal to $V$ and an angle $\vartheta \in\left(0, \frac{\pi}{2}\right)$ we then define the wedge with spine $V$, axis $\nu$ and opening angle $\vartheta$ as the set
$$
W(V, \nu, \vartheta):=\left\{y:\left|y-\mathbf{p}_V(y)-(y \cdot \nu) \nu\right| \leq(\tan \vartheta) y \cdot \nu\right\} .
$$
\end{definition}

Lemma 2.3 of \cite{delellis2021allardtype}, which we cite below, gives us the desired local geometrical information.
\begin{lemma}\label{l:convexbarrierlocal}Let $T$, $\sum_{i=1}^N Q_i\a{\Gamma_i}$, $\Omega$ be as in \ref{convexbarrier}. Then there is a $0<\vartheta<\frac{\pi}{2}$ (which is independent of the point in $\cup_{i=1}^N\Gamma_i$) such that the convex hull of $\cup_{i=1}^N\Gamma$ satisfies
$$
\textup{ConvexHull}(\cup_{i=1}^N\Gamma_i) \subset \bigcap_{q \in \Gamma}\left(q+W\left(T_q \Gamma, \nu(q), \vartheta\right)\right)
$$
\end{lemma}
\begin{proof}[Proof of Theorem \ref{t:densityexample}]
At points of $\Gamma^0 \setminus M$ or $\Gamma^1 \cap U \setminus M$, we apply a version of Lemma 2.3 of \cite{delellis2021allardtype}, which we just cited above Lemma \ref{l:convexbarrierlocal} obtaining: 
At a point $p$ of $\Gamma^0 \setminus M$ the tangent cone $C_p$ must satisfy
\begin{equation*}
\partial C_p=2\a{T_p(\Gamma^0)}
\end{equation*}
and 
\begin{equation*}
C_p \subset \left\{W(T_p \Gamma^0, \nu(p),\vartheta)\right\}.
\end{equation*}
At a point $p$ of $\Gamma^1 \cap U \setminus M$ the tangent cone must satisfy 
\begin{equation*}
\partial C_p=\a{T_p(\Gamma^1)}
\end{equation*}
and 
\begin{equation*}
C_p \subset \left\{W(T_p \Gamma^0, \nu(p),\vartheta)\right\}.
\end{equation*}

Thus for every point $p \in \Gamma^0 \setminus M$ or $p \in \Gamma^1 \cap U \setminus M$ we must have that the tangent cone must be of the convex barrier type. These are the convex barrier cones that we classified in Theorem 4.8 \cite{ian2024uniqueness} for arbitrary multiplicity, and Allard classified them for multiplicity $1$. In the first cases it is an open book and thus $\Theta(T,p)=1$, in the second case it is a single half-plane and thus $\Theta(T,p)=1/2$.

Let $p \in M$, $v_0$, $v_1$ be the normal vectors to $\Gamma^0$ and $\Gamma^1$ respectively. At $M$ the arguments from Lemma 2.3 \cite{delellis2021allardtype}, imply that 
\begin{equation*}
C_p \subset \left\{y:|y-\mathbf{p}_{T_p(M)}(y)-(y \cdot v_0)^{+}v_0 - (y\cdot \nu(p))|  \leq \tan (\theta)(y \cdot \nu(p))\right\}.
\end{equation*}
This implies that $C_p$ is of the type of cones we classified in Theorem \ref{thm:corneredconvex} and thus $C_p$ is a cornered open book. Thus $C_p$ is as in Figure \ref{img:nicecorneredopenbook} and $\Theta(T,p)=1.$
The regularity at $M$ follows from \ref{t:decomposition} and the Hölder continuity of the normal follows from Corollary \ref{c:wholeboundaryHölder}.
\end{proof}

We now turn to the second part: the conclusion.
\begin{theorem}Let $T, \Gamma^0, \Gamma^1$ satisfy the topological Assumption \ref{a:topologicalhypothesis}. Then $T$ has an essential boundary singularity at some point of $\Gamma^0$.
\end{theorem}
\begin{proof}[Proof of \ref{T:exampledetail}]
Since $\Gamma^0$ is diffeomorphic to a disk with boundary $M$ then there exist $\varphi: \left \{z \in \RR^2: |z| \leq 1\right\} \rightarrow \RR^{m+n}$ which is a smooth parameterization of $\Gamma^0$ and $\varphi(\mathbb{S}^1)=M.$

Then we can define the normal as a function of the disk as 
\begin{equation*}
\eta \circ \varphi: \left\{z \in \RR^2: |z| \leq 1\right\} \rightarrow \mathcal{A}_2(N\Gamma^0). 
\end{equation*}
We must have that at $M$, $\eta(p)=\a{v_1^p}+\a{-v_1^p}$ where $v_1^p$ is one of the normals of $\Gamma^1$ at $M$. Since $\Gamma^1$ is a Möbius strip then the graph of $\eta \circ \varphi \res \mathbb{S}^1$ induces a single continuous connected curve.

Since the normal bundle of $\Gamma^0$ is trivial, up to a diffeomorphism it agrees with the normal bundle of 
\begin{equation*}
    \left\{(x,y) \in \RR^2 \times \RR^n: |x|\leq 1, y=0 \right\} \subset \RR^{2+n}.
\end{equation*}After changing the image space by a diffeomorphism we conclude the proof with a classical elementary observation, which we collect in the following proposition:
\begin{proposition}\label{p:essentialsingularity}Let $f:\left\{z \in \CC: |z| \leq 1\right\} \rightarrow A_2(\RR^n)$ be a continuous multivalued function such that $f \res S^1$ is a continuous connected curve of multiplicity $1$. Then $f$ has an essential singularity.
\end{proposition}
\begin{proof}
This is a fairly classical consequence of the theory of covers in algebraic topology and that the boundary loop is non trivial.

For the reader’s convenience, we sketch the argument. Suppose, for contradiction, that $f$ has no interior essential singularities. Then, at every point $p$, $f$ admits a continuous local selection into two single-valued branches $f_1$ and $f_2$. This remains true at boundary points by continuity, since there are no double points at the boundary. By compactness, we can cover the closed disk with finitely many open sets over which such selections exist. If a continuous selection exists for two overlapping open sets $U$ and $V$ then they must agree on $U \cap V$, and thus we can extend the selection to  $U \cup V$. Iterating this, we obtain a global continuous selection on the whole disk. But this leads to a contradiction: the trace of $f$ on the boundary is a connected loop of multiplicity 1, which cannot be written as the disjoint union of two continuous branches. Therefore, $f$ must have at least one interior essential singularity.
\end{proof}

The above proposition implies that $f$ must have an essential singularity, which in turn forces $T$ to have an essential boundary singularity at some interior point of $\Gamma^0$. Otherwise, $f$ would have no interior essential singularities, which contradicts the proposition. 
\end{proof}

\section{Essential singularities for the linear problem}In the introduction, we sketched the validity of a general result for Dir-minimizing functions. We point their validity in this section without giving a proof, since the main interest of this work is the area-minimizing setting and the proof of the result for the linear problem is analogous to those for the nonlinear case.

\begin{theorem}Let $\Omega$ be a compact domain in $\RR^3$ which is $C^2$-diffeomorphic to a cylinder. Thus $\partial \Omega$ consists of three two dimensional surfaces, $D_0, D_1, C$, where  $D_0$ and $D_1$ are disks and $C$ is a cylinder which connects $\partial D_0$ with $\partial D_1$. Let $u \in W^{1,2}(\overline{\Omega},A_2(\RR^n))$. 
Assume that:
\begin{itemize}
\item $C$ meets $D_0$ orthogonally.
\item The function $u \res C$ is $C^2$ in a neighborhood of $\partial D_0$, and it is constantly equal to $2\a{0}$ on $D_0.$
\item Assume that the normal derivative of $u \res C$ to $\partial D_0$ is a continuous connected curve of multiplicity $1$.
\end{itemize}
Then $u$ has an essential boundary singularity at an interior point of the $2$-disk $D_0$.
\end{theorem}
\begin{corollary}\label{c:singularlinearproblem}There exists $u: W^{1,2}(\RR^2 \times \RR^+) \rightarrow \mathcal{A}_2(\RR^2)$ which is $I>1$ homogeneous Dir-minimizer.
\end{corollary}
\begin{proof}[Proof of Corollary \ref{c:singularlinearproblem}]
We consider $u$ as in the introduction or $u$ as above with $\Omega$ the unit cylinder. Then, there exists a point $p$ with an essential boundary singularity of $u$. If the frequency at $p$ is $1$, then, by Theorem 10.3 \cite{ian2024uniqueness} there is a unique linear blowup and thus a neighborhood of $p$, $u$ can be separated as two classical harmonic functions. This is a contradiction since at $p$, $u$ has an essential singularity.

Thus, the tangent map at $p$ is an $I>1$ homogeneous Dir-minimizer in the half-space which allows us to conclude.
\end{proof}
\bibliographystyle{alpha}
\addtocontents{toc}{\protect\enlargethispage{\baselineskip}}
\bibliography{biblio}
\end{document}